\documentclass[11pt]{article}

\usepackage{amssymb,latexsym,amsmath,enumerate,verbatim,amsfonts,amsthm}
\usepackage{color} 
\usepackage[active]{srcltx}
\usepackage{cite}

\newtheorem{lemma}{Lemma}[section]
\newtheorem{definition}{Definition}[section]
\newtheorem{assumption}{Assumption}[section]

\newtheorem{theorem}{Theorem}[section]
\newtheorem{corollary}{Corollary}[section]
\newtheorem{proposition}{Proposition}[section]
\newtheorem{remark}{Remark}[section]

\numberwithin{equation}{section}

\theoremstyle{definition}
\newtheorem{example}{Example}[section]

\newcommand{\ba}{\begin{array}}
\newcommand{\ea}{\end{array}}
\newcommand{\beq}{\begin{equation}}
\newcommand{\eeq}{\end{equation}}
\newcommand{\beqa}{\begin{eqnarray}}
\newcommand{\eeqa}{\end{eqnarray}}
\newcommand{\beqas}{\begin{eqnarray*}}
\newcommand{\eeqas}{\end{eqnarray*}}
\newcommand{\bi}{\begin{itemize}}
\newcommand{\ei}{\end{itemize}}
\newcommand{\gap}{\hspace*{2em}}

\def\vgap{\vspace*{.1in}}

\def\eqref#1{(\ref{#1})}

\def\Argmin{{\rm Argmin}}

\def\QED{\ifhmode\unskip\nobreak\fi\ifmmode\ifinner\else\hskip5pt\fi\fi
  \hbox{\hskip5pt\vrule width5pt height5pt depth1.5pt\hskip1pt}}

\def\cS{{\cal S}}
\def\dist{{\rm dist}}

\def\cB{{\bf B}}
\def\cK{{\cal K}}
\def\cN{{\cal N}}
\def\cS{{\cal S}}

\def\eps{{\epsilon}}

\def\xfeas{{x^{\rm feas}}}

\def\cS{{\Omega}}
\def\cU{{\mathcal{U}}}

\def\R{{\rm I\!R}}

\def\Argmin{\mathop{\rm Arg\,min}}

\title{\sf Penalty methods for a class of non-Lipschitz optimization problems}

\author{Xiaojun Chen\thanks{Department of Applied Mathematics, The Hong Kong Polytechnic
University, Hong Kong, China.  E-mail: {\tt maxjchen@polyu.edu.hk}.  This author's work is supported partly by Hong Kong Research Grant Council grant PolyU5001/12p.} \and Zhaosong Lu\thanks{Department of Mathematics, Simon Fraser University, Burnaby, BC, V5A 1S6, Canada. Email: {\tt zhaosong@sfu.ca}. This author's work is supported in part by NSERC Discovery Grant.} \and Ting Kei Pong\thanks{Department of Applied Mathematics, The Hong Kong Polytechnic
University, Hong Kong, China.  E-mail: {\tt tk.pong@polyu.edu.hk}. This author's work is supported partly by Hong Kong Research Grants Council PolyU253008/15p.} }

\date{June 26, 2015 (Revised: April 3, 2016)}

\begin{document}
\maketitle

\begin{abstract}
We consider a class of constrained optimization problems with a possibly nonconvex non-Lipschitz objective and a convex feasible set being the intersection of a polyhedron and a possibly degenerate ellipsoid. Such problems have a wide range of applications in data science, where the objective is used for inducing sparsity in the solutions while the constraint set models the noise tolerance and incorporates other prior information for data fitting. To solve this class of constrained optimization problems, a common approach is the penalty method. However,
there is little theory on exact penalization for problems with nonconvex and non-Lipschitz objective functions. In this paper, we study the existence of exact penalty parameters regarding local minimizers, stationary points and $\epsilon$-minimizers under suitable assumptions. Moreover, we discuss a penalty method whose subproblems are solved via a nonmonotone proximal gradient method with
a suitable update scheme for the penalty parameters, and prove the convergence of the algorithm to a KKT point of the constrained problem. Preliminary numerical results demonstrate the efficiency of the penalty method for finding sparse solutions of underdetermined linear systems.

\end{abstract}


 \hspace{0.8in} nonconvex optimization, non-Lipschitz optimization.


\section{Introduction}

  We consider the following constrained optimization problem:
  \begin{equation}\label{eq:P0}
    \begin{array}{rl}
     \min\limits_x & \Phi(x)\\
      {\rm s.t.} & x \in S:= S_1\cap S_2,
    \end{array}
  \end{equation}
  where $\Phi:\R^n\to\R$ is a nonnegative continuous function, $S_1\subseteq \R^n$ is a simple polyhedron, and
  $$S_2=\{x :\;  \|Ax-b\| \le \sigma, \,\,  \, \,  Bx\le h \, \}. $$
        Here $A\in \R^{m \times n}, b\in \R^m, \sigma \ge0, B\in \R^{\ell \times n}$ and $h\in \R^\ell$
       are given matrices and vectors.
  We emphasize that $\Phi$ is neither necessarily convex nor locally Lipschitz continuous. To avoid triviality, we suppose that the feasible region $S$ is nonempty.

  Problem \eqref{eq:P0} is flexible enough to accommodate a wide range of optimization models with important applications in imaging sciences, signal
  processing, and statistical variable selections, etc. For example,
  with $S_1=\R^n$ and $B$ being vacuous, i.e., $S = S_2 = \{x:\; \|Ax - b\|\le \sigma\}$, problem \eqref{eq:P0} reduces to the following problem
  \begin{equation}\label{eq:P00}
    \begin{array}{rl}
     \min\limits_{x} & \Phi(x)\\
      {\rm s.t.} & \|Ax-b\|\le \sigma.
    \end{array}
  \end{equation}
  This problem with $\Phi(x)=\|x\|_1$ has been studied extensively for recovering sparse signals from the possibly noisy measurements $b$; here, the parameter $\sigma$ allows the user to {\em explicitly} specify the tolerance for the noise level.
We refer the readers to the comprehensive review \cite{BDE} for more details.
In addition, we emphasize that the objective function $\Phi$ in our model \eqref{eq:P0} is allowed to be nonsmooth and possibly nonconvex non-Lipschitz. This enables the choice of various objective functions for inducing desirable structures in the optimal solutions. For instance, when sparsity is of concern, one popular choice of $\Phi$ is $\Phi(x)=\sum^n_{i=1} \phi(x_i)$, with $\phi$ being one of the widely used penalty functions, such as the bridge penalty \cite{HuangHorowitzMa08,KnightFu00}, the fraction penalty \cite{GemanReynolds92} and the logistic penalty \cite{NiNgZhCh08}.
On the other hand, we note that the simple polyhedron $S_1$ can be used for incorporating hard constraints/prior information that must be satisfied by the decision variables in applications. For instance,
if a true solution to \eqref{eq:P00} is known to be in a certain interval $[l, u]$ for some $l < u$, $l$ and $u\in \R^n$,
then the $S_1$ can be chosen to be $[l,u]$ instead of just $\R^n$.
Constraints of this kind arise naturally in applications such as image restoration, where all gray level images have intensity values ranging from $0$ to $1$.
As shown in \cite{Beck-Teboulle,Chan-Tao,Ng-Weiss-Yuan}, incorporating the bound constraints can lead to substantial
improvements in the quality of the restored image.

While \eqref{eq:P0} is a very flexible model covering a wide range of applications, this optimization problem is a constrained optimization problem, which is typically hard to solve. In the case when $\Phi$ is convex, $S_1 = \R^n$ and $B$ is vacuous, i.e., \eqref{eq:P00}, it is well known that the problem is equivalent to solving
\begin{equation}\label{eq:P-1}
      \min_x \ H_\lambda(x) := \lambda \|Ax - b\|^2 + \Phi(x)
  \end{equation}
for some regularization parameter $\lambda > 0$, under some mild assumptions; see, for example, \cite{FTse07}. Unlike \eqref{eq:P00}, for many choices of $\Phi$, the regularized formulation \eqref{eq:P-1} can be solved readily by various first-order methods such as the NPG method in \cite{WrNoFi09}. This regularized formulation has been extensively studied in {\em both} cases where $\Phi$ is convex or nonconvex in the last few decades; see, for example,
\cite{BDE,XJChen12,CGWY,CXY,JFan,FanJ,GJY,GemanReynolds92,GHLYZ13,HuangHorowitzMa08,Lu,NiNgZhCh08,Rock98,YLHX14,CHZhang,KnightFu00}. Nevertheless, the equivalence between \eqref{eq:P00} and  \eqref{eq:P-1} does not hold in the nonconvex scenario: indeed, for nonconvex $\Phi$ and certain data ($A,b,\sigma$), there does not exist a
$\lambda$ so that problems \eqref{eq:P00} and \eqref{eq:P-1} have a common global or local minimizer; see our Example~\ref{example1}. In particular, one cannot solve \eqref{eq:P00} via solving the unconstrained problem \eqref{eq:P-1} for a suitable $\lambda$ in general.

In a hope of constructing a simpler optimization problem whose local/global minimizers are closely related to \eqref{eq:P0}, we resort to the penalty approach. 
While this is a standard approach, there are two important new ingredients in our work. First, although exact penalization for constrained optimization problems with a Lipschitz objective has been well studied (see, for example, \cite{Nocedal}), to the best of our knowledge, there is little theory and development for problems with nonconvex non-Lipschitz objectives such as problem \eqref{eq:P0} with $\phi$ being the bridge penalty. Second, we consider partial penalization that keeps the constraints $S_1$ in \eqref{eq:P0}. Recall that the set $S_1$ in \eqref{eq:P0} can be used to model {\em hard} constraints that must be satisfied or simple constraints that can be easily satisfied\footnote{This means that the projection onto $S_1$ is easy to compute.}, while the set $S_2$ can be used to model {\em soft} constraints that only need to be approximately satisfied. Consequently, it can be advantageous to be able to penalize only the constraints corresponding to $S_2$ and {\em keep} the hard constraints $S_1$. 

The penalty problem we consider is
\begin{equation}\label{eq:P1}
  \min_{x\in S_1}  \ F_\lambda(x) := \lambda [ (\|Ax - b\|^2 - \sigma^2)_+ + \|(Bx-h)_+\|_1] + \Phi(x)
\end{equation}
for some $\lambda > 0$,  where $a_+$ denotes the vector whose $i$th entry is $\max\{a_i,0\}$ 
for any $a\in\Re^n$. 
In this paper, we derive various (partial) exact penalization results regarding \eqref{eq:P0} and \eqref{eq:P1}.
 Specifically,
under some suitable assumptions, we establish that:
\bi
\item[(i)] any local minimizer of problem \eqref{eq:P0} is also that of problem \eqref{eq:P1},  provided that $\lambda \ge \lambda^*$ for some $\lambda^*>0$;
\item[(ii)] any global minimizer of problem \eqref{eq:P0} is an $\epsilon$-global minimizer of problem \eqref{eq:P1}, provided that $\lambda \ge \lambda^*$ for
some $\lambda^*>0$;
\item[(iii)] the projection of any global minimizer of problem \eqref{eq:P1} onto the feasible set $S$ of
 problem \eqref{eq:P0} produces an $\epsilon$-global minimizer of problem
\eqref{eq:P0}, provided that $\lambda \ge \lambda^*$ for some $\lambda^*>0$.
\ei
Consequently, problem \eqref{eq:P1} is an exact penalty formulation for \eqref{eq:P0}, and an approximate solution of problem \eqref{eq:P0} can be obtained
by solving \eqref{eq:P1} with $\lambda=\lambda^*$ if an exact penalty parameter $\lambda^*$ is known.

In practice, the value of such $\lambda^*$ is, however, generally unknown. Owing to this, we further propose a penalty method for solving \eqref{eq:P0} whose subproblems are
(partially) smoothed and then solved approximately via a nonmonotone proximal gradient (NPG) method \cite{WrNoFi09} with a suitable update scheme for the penalty and smoothing parameters.
It is noteworthy that the NPG method originally studied in \cite{WrNoFi09} was proposed for minimizing the sum of a possibly nonsmooth function and a smooth function whose
gradient is {\it globally} Lipschitz continuous. Nevertheless, the gradient of the smooth component associated with our subproblems is {\it locally} but not globally Lipschitz continuous. We
are fortunately able to show that this NPG method is indeed capable of solving a more general class of problems which includes our subproblems as a special case. In addition, we show that
any accumulation point of the sequence generated by our penalty method is a KKT point of \eqref{eq:P0} under suitable assumptions. Finally, to benchmark our approach, we consider a sparse recovery problem and compare \eqref{eq:P00} with $\Phi(x) = \sum_{i=1}^n|x_i|^{\frac 12}$ solved by our penalty method against two other approaches: solving \eqref{eq:P00} with $\Phi(x) = \|x\|_1$ by the SPGL1 \cite{BergMPF08}
for finding sparse solutions, and solving \eqref{eq:P-1} with $\Phi(x) = \sum_{i=1}^n|x_i|^{\frac 12}$ for a suitably chosen $\lambda$.
Our numerical results demonstrate that the solutions produced by our approach are sparser and have smaller recovery errors than those found by the other approaches.

The rest of the paper is organized as follows. We present notation and preliminary materials in Section~\ref{sec2}. In Section~\ref{sec4}, we study  the existence of
exact penalty parameters regarding local minimizers and $\epsilon$-minimizers.  In Section~\ref{sec5}, we discuss the
first-order optimality conditions for problems \eqref{eq:P0} and \eqref{eq:P1}. We then propose a penalty method for solving problem  \eqref{eq:P0}
with an update scheme for the penalty parameters and establish its convergence to KKT points of \eqref{eq:P0}.  In Section~\ref{sec6}, we conduct numerical experiments to test the performance
of our method in sparse recovery. Concluding remarks are given in Section~\ref{sec7}.

\section{Notation and preliminaries}\label{sec2}

  We use $\R$ and $\R^n$ to denote the set of real numbers and the $n$-dimensional Euclidean space. For any $x\in \R^n$, let $x_i$ denote the $i$th entry of $x$, and ${\rm Diag}(x)$ denote the diagonal matrix whose $i$th diagonal entry is $x_i$, respectively. We denote the Euclidean norm of $x$ by $\|x\|$, the $\ell_1$ norm by $\|x\|_1$, the infinity norm (sup norm) by $\|x\|_\infty$, and the $p$ quasi-norm by $\|x\|_p := \left(\sum_{i=1}^n|x_i|^p\right)^\frac1p$, for any $0<p<1$. Moreover, we let $|x|^p$ denote the vector whose $i$th entry is $|x_i|^p$ and $(x)_+$ denote the vector whose $i$th entry is $\max\{x_i,0\}$.
  Given an index set $I\subseteq \{1,\ldots,n\}$, let $\bar I$ denote the complement of $I$. For any vector $x$, we write $x_I\in \R^{|I|}$ to denote the restriction of $x$ onto $I$. We also denote by $A_I$ the matrix formed from a matrix $A$ by picking the columns indexed by $I$. In addition, we use $\ker A$ to denote the null space of a matrix $A$.

  For any closed set $D$, we let ${\rm dist}(x,D) = \inf_{y\in D}\|x - y\|$ denote the distance from $x$ to $D$, and ${\rm conv}(D)$ denote the convex hull of $D$. We let $P_D(x)$ denote the set of closest points in $D$ to $x\in \R^n$; this reduces to a singleton if $D$ is additionally convex. For a closed convex set $D$, the normal cone at $x\in D$ is defined as
  \[
  \cN_D(x) := \{y:\; y^T(u - x)\le 0\ \ \forall u\in D\}.
  \]
  The indicator function is denoted by $\delta_D$, which is the function that is zero in $D$ and is infinity elsewhere.
  Finally, we let $\cB(a; r)$ denote the closed ball of radius $r$ centered at $a$, i.e., $\cB(a; r) = \{x \in \R^n:\;\|x-a\| \le r \}$.

  We recall from \cite[Definition~8.3]{Rock98} that for a proper lower semicontinuous function $f$, the (limiting) subdifferential and horizon subdifferential are defined respectively as
  \begin{equation*}
  \begin{split}
\partial f(x)&:=\left\{v:\;\exists x^k \stackrel{f}{\to} x,\;v^k\to v\;\mbox{ with }\liminf_{z \to x^k}\frac{f(z)-f(x^k)-\langle v^k,z-x^k\rangle}{\|z-x^k\|}\ge 0\ \forall k\right\},\\
\partial^{^\infty}\!\!f(x)&:=\left\{v :\;\exists x^k \stackrel{f}{\to} x,\;\lambda_kv^k\to v, \lambda_k\downarrow 0\;\mbox{ with }\liminf_{z \to x^k}\frac{f(z)-f(x^k)-\langle v^k,z-x^k\rangle}{\|z-x^k\|}\ge 0\ \forall k\right\},
\end{split}
\end{equation*}
where $\lambda_k\downarrow 0$ means $\lambda_k > 0$ and $\lambda_k\to 0$, and $x^k  \stackrel{f}{\to} x$ means both $x^k \to x$ and $f(x^k)\to f(x)$.
It is well known that the following properties hold:   
\begin{equation}\label{outersemi}
\begin{split}
  &\left\{v:\; \exists x^k \stackrel{f}{\to} x,\; v^k\to v\;, v^k\in \partial f(x^k)\right\} \subseteq \partial f(x),\\
  &\left\{v:\; \exists x^k \stackrel{f}{\to} x,\; \lambda_k v^k\to v\;,\lambda_k\downarrow 0\;, v^k\in \partial f(x^k)\right\} \subseteq \partial^{^\infty}\!\! f(x).
\end{split}
\end{equation}
Moreover, if $f$ is convex, the above definition of subdifferential coincides with the classical subdifferential in convex analysis \cite[Proposition~8.12]{Rock98}. Furthermore, for a continuously differentiable $f$, we simply have $\partial f(x) = \{\nabla f(x)\}$, where $\nabla f(x)$ is the gradient of $f$ at $x$ \cite[Exercise~8.8(b)]{Rock98}. We also use $\partial_{x_i} f(x)$ to denote the subdifferential with respect to the variable $x_i$. Finally, when $\Phi(x) = \sum_{i=1}^n\phi(x_i)$ for some continuous function $\phi$, we have from \cite[Proposition~10.5]{Rock98} that
\begin{equation}\label{productrule}
  \partial \Phi(x) = \partial \phi(x_1)\times \partial \phi(x_2)\times \cdots \times \partial \phi(x_n).
\end{equation}

For the convenience of readers, we now state our blanket assumptions on \eqref{eq:P0} explicitly here for easy reference.
\begin{assumption}[{{\bf Blanket assumptions on \eqref{eq:P0}}}]
  Throughout this paper, $\Phi$ is a nonnegative continuous function. The feasible set of \eqref{eq:P0} is $S:= S_1\cap S_2$, where $S_1$ is a simple polyhedron given by $\{x:\; Dx \le d\}$, and
  \[
  \quad S_2=\{x :\;  \|Ax-b\| \le \sigma, \,\,  \, \,  Bx\le h \, \}.
  \]
  Moreover, $A$ has full row rank and there exists $x_0\in S$ so that $\|Ax_0 - b\|< \sigma$.
\end{assumption}

We next present some auxiliary lemmas. The first lemma is a well-known result on error bound concerning $S_1$ and $S_2$, obtained as an immediate corollary of \cite[Theorem~3.1]{LuoLuo94}.

\begin{lemma}\label{lem0}
  There exists a $C > 0$ so that for all $x\in \R^n$, we have
  \[
  {\rm dist}(x,S) \le C\left[(\|Ax - b\|^2 - \sigma^2)_+ + \|(Bx - h)_+\|_1 + \|(Dx - d)_+\|_1\right].
  \]
  Consequently, for any $x\in S_1$, we have
  \begin{equation}\label{err-bdd0}
  {\rm dist}(x,S) \le C\left[(\|Ax - b\|^2 - \sigma^2)_+ + \|(Bx - h)_+\|_1\right].
  \end{equation}
\end{lemma}

The constant $C$ in the above lemma cannot be explicitly computed in general. We next present a more explicit representation of this constant in some special cases. We start with the case where $S_1 = \R^n$ and $B$ is vacuous, i.e., $S = S_2 = \{x:\; \|Ax - b\|\le \sigma\}$.

  \begin{lemma}\label{lem-1}
    Suppose that $S = S_2 = \{x:\; \|Ax - b\|\le \sigma\}$. Then there exists a $C > 0$ so that for all $x$,
    \begin{equation*}
      {\rm dist}(x,S)\le \|A^\dagger\| (\|Ax - b\| - \sigma)_+\le C (\|Ax - b\|^2 - \sigma^2)_+.
    \end{equation*}
    Indeed, $C$ can be chosen to be $\frac{\|A^\dagger\|}{\sigma}$, where $A^\dagger = A^T(AA^T)^{-1}$ is the pseudo-inverse of $A$.
  \end{lemma}
  \begin{proof}
    Notice that $S = A^\dagger\cB(b;\sigma) + \ker A$. Moreover, for any $x$, $x - A^\dagger Ax \in \ker A$. Thus, we have
    \begin{equation*}
      \begin{split}
        &{\rm dist}(x,S) = {\rm dist}(A^\dagger Ax + [x - A^\dagger Ax], A^\dagger\cB(b;\sigma) + \ker A)\\
        &\le {\rm dist}(A^\dagger Ax, A^\dagger\cB(b;\sigma)) \le \|A^\dagger\| {\rm dist}(Ax, \cB(b;\sigma)) = \|A^\dagger\| (\|Ax - b\| - \sigma)_+,
      \end{split}
    \end{equation*}
    where the last equality follows from a direct computation based on the fact that the projection from any point $u\notin \cB(b;\sigma)$ onto $\cB(b;\sigma)$ is $b + \sigma\frac{u - b}{\|u - b\|}$. The conclusion of the lemma now follows from the above estimate and the following simple relation:
    \begin{equation*}
      (\|Ax - b\| - \sigma)_+ = \left(\frac{\|Ax - b\|^2 - \sigma^2}{\|Ax - b\| + \sigma}\right)_+ \le \frac{1}{\sigma}(\|Ax - b\|^2 - \sigma^2)_+.
    \end{equation*}
  \end{proof}

We next consider the case where $S$ is compact. We refer the readers to \cite[Lemma~3.2.3]{FchP03:vi} and \cite[Remark~3.2.4]{FchP03:vi} for an explicit finite upper bound for the constant $\beta$ in \eqref{bd4} below.
%
  \begin{lemma}\label{lem1}
  Suppose there exist $x_s\in S$, $R > \delta > 0$ so that $\sup_{u\in \cB(x_s;\delta)}\|Au - b\|\le \sigma$ and $S\subseteq \cB(x_s;R)$.
  Then there exists $\beta > 0$ so that for all $x\in \R^n$, we have
  \begin{equation}\label{bd4}
  {\rm dist}(x,S) \le 2\left(1 + \frac{R}{\delta}\right)\left(\frac{\|A^\dagger\|}{\sigma}(\|Ax - b\|^2 - \sigma^2)_+ + \beta\left\|{\begin{pmatrix}
    Bx - h\\ Dx - d
  \end{pmatrix}}_+\right\|_1\right).
  \end{equation}
\end{lemma}
\begin{proof}
  Let $\cS_1 = \{x:\; Bx \le h,\ Dx \le d\}$ and $\cS_2=\{x: \|Ax - b\|\le \sigma\}$. Then $S = \cS_1\cap \cS_2$. From the assumptions and \cite[Lemma~2.1]{LiMaPong14} (see also \cite[Lemma~4.10]{LiNgPong07}), we see that for all $x\in \R^n$, we have
  \begin{equation}\label{bd2}
    \dist(x,S) \le 2\left(1 + \frac{R}{\delta}\right)\max\{\dist(x,\cS_1),\dist(x,\cS_2)\}.
  \end{equation}
The desired conclusion now follows from \eqref{bd2}, Lemma~\ref{lem-1} and \cite[Lemma~3.2.3]{FchP03:vi} (Hoffman error bound).
\end{proof}

  We end this section with the following auxiliary lemmas concerning the function $t\mapsto t^p$, $0<p<1$.

  \begin{lemma}\label{lem2}
    Let $0<p < 1$. For any nonnegative numbers $s$ and $t$, it holds that
    \[
    |s^p - t^p| \le |s - t|^p.
    \]
  \end{lemma}
  \begin{proof}
    Without loss of generality, we may assume that $s \ge t$.
    Consider $h(r) := 1 - r^p - (1-r)^p$ for $r\in [0,1]$.
    Simple differential calculus shows that $h(r) \le h(0) = 0 = h(1)$ whenever $r \in [0, 1]$. The desired conclusion then
follows by setting $r = \frac{t}{s}$.
  \end{proof}

  \begin{lemma}\label{lem:lp}
    Let $0<p<1$. Then the following statements hold.
    \begin{enumerate}[{(i)}]
      \item Let $h(t) = |t|^p$. Then $\partial h(0) = \partial^{^\infty}\!\!h(0) = \R$.
      \item Let $H(x) = \sum_{i=1}^n|x_i|^p$ and fix any $x^*\in \R^n$. Let $I:= \{i:\; x^* \neq 0\}$. Then
      \[
      \partial^{^\infty}\!\!H(x^*) = \{v:\; v_i = 0\  {\rm for}\  i\in I\}.
      \]
    \end{enumerate}
  \end{lemma}
  \begin{proof}
    We first prove (i). Consider the set
    \[
    \hat\partial h(0) := \left\{s\in \R:\; \liminf_{t\to 0}\frac{|t|^p - st}{|t|} \ge 0\right\}.
    \]
    Since $\liminf_{t\to 0}|t|^{p-1} = \infty$ due to $0<p<1$, we see immediately that $\hat\partial h(0) = \R$. Since we have from \cite[Theorem~8.6]{Rock98} that $\hat \partial h(0)\subseteq \partial h(0)$ and that $\partial^{^\infty}\!\!h(0)$ contains the recession cone of $\hat\partial h(0)$, we conclude further that $\partial h(0) = \partial^{^\infty}\!\!h(0) = \R$.

    We next prove (ii). Part (i) together with the fact that $\hat\partial h(0) = \R$ and \cite[Corollary~8.11]{Rock98} shows that $h(t)= |t|^p$ is regular at $0$. In addition, $h$ is clearly regular at any $t\neq 0$. Then, according to \cite[Proposition~10.5]{Rock98}, we have
    \[
    \partial^{^\infty}\!\!H(x^*) = \partial^{^\infty}\!\!h(x_1^*)\times \cdots \partial^{^\infty}\!\!h(x_n^*),
    \]
    from which the conclusion follows immediately.
  \end{proof}

  \section{Exact Penalization}\label{sec4}

  Problem \eqref{eq:P0} is a constrained optimization problem, which can be difficult to solve when the constraint set $S$ is complicated. In the case when $\Phi$ is convex, $\sigma > 0$, $S_1 = \R^n$ and $B$ is vacuous, i.e., \eqref{eq:P00}, it is well known that the problem is equivalent to solving the unconstrained optimization problem \eqref{eq:P-1} for some suitable $\lambda > 0$; see, for example, \cite{FTse07}. However, as we will illustrate in the next example, this is no longer true for a general nonconvex $\Phi$.

  \begin{example}\label{example1}
    Consider the following one-dimensional optimization problem:
    \begin{equation}\label{ex:P0}
      \begin{array}{rl}
        \min\limits_t & \phi(t)\\
        {\rm s.t.} & |t - a|\le \gamma a
      \end{array}
    \end{equation}
    for some $a > 0$ and $\gamma\in (0,1)$. Assume that $\phi$ is strictly increasing on $[0,\infty)$.

    It is clear that $t^* = (1 -\gamma)a$ is the global minimizer of \eqref{ex:P0}. Suppose that $\phi$ is twice continuously differentiable at $t^*$. Then it is easy to check from the first-order optimality condition that $t^*$ is a stationary point of
    \begin{equation}\label{ex:P1}
      \min_t\ \lambda (t - a)^2 + \phi(t)
    \end{equation}
    only when $\lambda = \phi'(t^*)/(2\gamma a)$, which is nonnegative since $\phi$ is monotone. Next, the second derivative of the objective of \eqref{ex:P1} with $\lambda = \phi'(t^*)/(2\gamma a)$ at $t^*$ is given by
    \begin{equation}\label{2ndderiv}
    2\lambda + \phi''\left(t^*\right) = \frac{\phi'(t^*)}{\gamma a} + \phi''\left(t^*\right).
    \end{equation}
    If this quantity is negative, then $t^*$ cannot be a local minimizer of \eqref{ex:P1} even for $\lambda = \phi'(t^*)/(2\gamma a)$, and consequently, $t^*$ cannot be a local minimizer of \eqref{ex:P1} for any $\lambda > 0$.

    Some concrete examples of $\phi$ and $a$ such that \eqref{2ndderiv} is negative are given below, where the $\phi$'s are building blocks for widely used nonconvex regularization functions.
    \begin{enumerate}
      \item bridge penalty $\phi(t) = |t|^p$ for $0<p<1$ \cite{HuangHorowitzMa08,KnightFu00}.\\
       For any $a >0$, \eqref{2ndderiv} equals $p(t^*)^{p-2}\left(p-2+1/{\gamma}\right)$. Hence, \eqref{2ndderiv} is negative if $p < 2 - 1/{\gamma}$. Since $p$ is positive, this can happen when $\gamma > 1/(2-p)$;
      \item fraction penalty $\phi(t) = \alpha|t|/(1 + \alpha|t|)$ for $\alpha > 0$ \cite{GemanReynolds92}.\\
       For any $a > 0$, a direct computation shows that \eqref{2ndderiv} equals $({\alpha}/{\gamma a})\left(1+\alpha t^*\right)^{-3}[1 + (1-3\gamma)\alpha a]$, which is negative when $1 + (1-3\gamma)\alpha a < 0$. Since $a$ and $\alpha$ are both positive, this can happen when $\gamma > (1+\alpha a)/(3\alpha a)$;
      \item logistic penalty $\phi(t) = \log(1 + \alpha|t|)$ for $\alpha > 0$ \cite{NiNgZhCh08}.\\
        For any $a > 0$, \eqref{2ndderiv} equals $({\alpha}/{\gamma a})\left(1+\alpha t^*\right)^{-2}\left[ 1 + (1-2\gamma)\alpha a\right]$, which is negative if $1 + (1-2\gamma)\alpha a < 0$. Since $a$ and $\alpha$ are both positive, this can happen when $\gamma > (1+\alpha a)/(2\alpha a)$.
    \end{enumerate}

  \end{example}

  Example~\ref{example1} shows that the negativity of  $\phi''$ prevents us from building a relationship between  \eqref{eq:P00} and \eqref{eq:P-1} regarding global or local minimizers. 
  In general, we cannot always find a $\lambda$ such that the intersection of the sets of global (local) minimizers of \eqref{eq:P00} and \eqref{eq:P-1}
  is nonempty, when $\phi$ is monotone and concave in $[0,\infty)$.

  In order to build a simpler optimization problem whose local/global minimizers are related to the constrained problem \eqref{eq:P0} (which contains \eqref{eq:P00} as a special case) when $\Phi$ is possibly nonconvex, we adopt the penalty approach. We hereby emphasize again that there is little theory concerning the penalty approach when $\Phi$ is non-Lipschitz. Moreover, it is not common in the literature to consider {\em partial} penalization that keeps part of the constraints, $S_1$, in the penalized problem \eqref{eq:P1}.
  In this section, we shall study various (partial) exact penalization results concerning the problems \eqref{eq:P0} and \eqref{eq:P1}, for both locally Lipschitz and non-Lipschitz objectives $\Phi$.

  \subsection{A general penalization result}\label{sec40}

  We first present some results regarding exact penalty reformulation for a general optimization problem. These results will be applied in subsequent subsections
  to derive various exact penalization results. The following lemma is similar to \cite[Proposition~4]{LeDN12}. For self-contained purpose, we provide a simple proof.

\begin{lemma} \label{general-penalty}
Consider the problem
\begin{equation} \label{general-opt}
\min\limits_{x \in \cS_1 \cap \cS_2} f(x),
\end{equation}
where $\cS_1$ and $\cS_2$ are two nonempty closed sets in $\R^n$. Assume that $f$ is Lipschitz continuous in
$\cS_1$ with a Lipschitz constant $L_f>0$, and moreover, problem \eqref{general-opt} has at least one optimal
solution. Suppose in addition that there is a function $Q: \cS_1 \to \R_+$ satisifying
\begin{equation} \label{Q}
Q(x)  \ge \dist(x,\cS_1 \cap \cS_2) \quad \forall x \in \cS_1; \quad  Q(x) =  0 \quad \forall x \in \cS_1 \cap \cS_2.
\end{equation}

Then it holds that:
\begin{enumerate}
\item[(i)]
 if $x^*$ is a global minimizer of \eqref{general-opt}, then
$x^*$ is a global minimizer of
\begin{equation} \label{P1}
\min\limits_{x\in\cS_1} \ f(x) + \lambda Q(x)
\end{equation}
whenever $\lambda \ge L_f$;
\item[(ii)] if $x^*$ is a global minimizer of \eqref{P1} for some $\lambda > L_f$, then $x^*$ is a global
minimizer of \eqref{general-opt}.
\end{enumerate}
\end{lemma}

\begin{proof}
Since $f$ is Lipschitz continuous in $\cS_1$ with a Lipschitz constant  $L_f>0$, it follows that for all $\lambda \ge L_f$,
\[
f(x) + \lambda \ \dist(x,\cS_1\cap\cS_2) \ \ge \ f(y) \quad\quad \forall x \in \cS_1,\ \forall y \in P_{\cS_1\cap\cS_2}(x),
\]
which together with \eqref{Q} implies that for any $\lambda \ge L_f$,
\[
f(x) + \lambda Q(x) \ \ge \  f(y) \quad\quad \forall x \in \cS_1,\ \forall y \in P_{\cS_1\cap\cS_2}(x).
\]
Using this relation, one can observe that  for all $\lambda \ge L_f$,
\[
\begin{split}
\inf\limits_{x\in\cS_1} \left\{f(x) + \lambda Q(x)\right\} & \ge \inf\limits_{x\in\cS_1, y \in P_{\cS_1\cap\cS_2}(x)} f(y) \ = \ \inf\limits_{x\in\cS_1\cap\cS_2} f(x)\\
&= \inf_{x\in \cS_1\cap \cS_2} f(x) + \lambda Q(x) \ge \inf\limits_{x\in\cS_1} \left\{f(x) + \lambda Q(x)\right\},
\end{split}
\]
where the second equality follows from the fact that $Q(x)=0$ for all $x\in \cS_1 \cap \cS_2$. Statement (i) follows immediately from this relation.

We next prove statement (ii). Suppose that $x^*\in \cS_1$ is a global minimizer of \eqref{P1} for some
$\lambda > L_f$. Using this and $Q(x)=0$ on $\cS_1\cap \cS_2$, we have
\beq \label{fQ}
f(x^*) + \lambda Q(x^*) \le f(x),
\eeq 
for any $x \in \cS_1\cap\cS_2$. This together with \eqref{Q} implies that for any $x \in P_{\cS_1\cap\cS_2}(x^*)$,
\[
f(x^*) + \lambda \ \dist(x^*,\cS_1\cap\cS_2) \ \le \  f(x).
\]
Using this relation and Lipschitz continuity of $f$, one can obtain that for any $x \in P_{\cS_1\cap\cS_2}(x^*)$,
\[
\dist(x^*,\cS_1\cap\cS_2) \le \frac{1}{\lambda}(f(x)- f(x^*)) \ \le\   \frac{L_f}{\lambda} \|x-x^*\|
\ = \ \frac{L_f}{\lambda} \dist(x^*,\cS_1\cap\cS_2),
\]
which along with $\lambda>L_f$ yields $\dist(x^*,\cS_1\cap\cS_2)=0$, that is,
$x^*\in\cS_1\cap\cS_2$. In addition, by \eqref{fQ} and $Q(x^*) \ge 0$, one can see that 
$f(x) \ge f(x^*)$ for any $x \in \cS_1\cap\cS_2$. Hence, $x^*$ is a global minimizer of 
\eqref{general-opt}.
\end{proof}

\gap

We next state a result regarding  the local minimizers of problems
\eqref{general-opt} and \eqref{P1}, whose proof is similar to that of Lemma~\ref{general-penalty}
and thus omitted.

\begin{corollary} \label{local-penalty}
 Assume that $f$ is locally Lipschitz continuous in $\cS_1$ and $Q$ satisfies \eqref{Q}. Suppose that
$x^*$ is a local minimizer of \eqref{general-opt}. Then there exists a $\lambda^* > 0$ such that
$x^*$ is a local minimizer of \eqref{P1} whenever $\lambda \ge \lambda^*$.
\end{corollary}

  \subsection{When $\Phi$ is locally Lipschitz continuous}\label{sec41}

  In this subsection, we consider the case where $\Phi$ is locally Lipschitz continuous and derive the corresponding exact regularization results concerning models \eqref{eq:P0} and \eqref{eq:P1}. This covers a lot of regularization functions used in practice, including many difference-of-convex functions; see, for example, \cite{YLHX14,GHLYZ13}.

  Our main result  concerns local and global minimizers of models \eqref{eq:P0} and \eqref{eq:P1}.
  \begin{theorem}[{{\bf Local \& global minimizers}}]\label{thm1}
    Suppose that $\Phi$ is locally Lipschitz continuous in $S_1$ and $x^*$ is a local minimizer of \eqref{eq:P0}. Then there exists a $\lambda^* > 0$ such that $x^*$ is a local minimizer of \eqref{eq:P1} whenever $\lambda \ge \lambda^*$. If $\Phi$ is indeed globally Lipschitz continuous in $S_1$, then there exists a $\lambda^* > 0$ such that
    any global minimizer of \eqref{eq:P0} is a global minimizer of \eqref{eq:P1} whenever $\lambda \ge \lambda^*$; moreover, if $x^*$ is a global minimizer of \eqref{eq:P1} for some $\lambda > \lambda^*$, then $x^*$ is a global minimizer of \eqref{eq:P0}.
  \end{theorem}
  \begin{proof}
    From Lemma~\ref{lem0}, we see that there exists a $C>0$ so that for all $x\in S_1$, we have
    \begin{equation*}
      {\rm dist}(x,S) \le C \left[(\|Ax - b\|^2 - \sigma^2)_+ + \|(Bx - h)_+\|_1\right].
    \end{equation*}
    The first conclusion now follows immediately from Corollary~\ref{local-penalty} by setting $f(x) = \Phi(x)$, $Q(x) =  C \left[(\|Ax - b\|^2 - \sigma^2)_+ + \|(Bx - h)_+\|_1\right]$, $\cS_1 = S_1$ and $\cS_2 = S_2$, while the second conclusion follows from Lemma~\ref{general-penalty}.
  \end{proof}

  \begin{remark}\label{remXX}
  It is not hard to see from the proof of Theorem~\ref{thm1} that with an explicit error bound modulus $C > 0$ in \eqref{err-bdd0}, the $\lambda^*$ in the theorem can be chosen to be $CL$, where $L$ is a local (resp., global) Lipschitz constant  $\Phi$.
  \end{remark}

   In the next example, we present explicit exact penalty functions for problem~\eqref{ex:P0} with some specific choices of $\phi$.

  \begin{example}
    Notice that the fraction penalty function and the logistic penalty function considered in Example~\ref{example1} are (globally) Lipschitz continuous, and have a Lipschitz constant  $\alpha$. From Theorem~\ref{thm1} and Remark~\ref{remXX}, we conclude that any global minimizer of \eqref{ex:P0} is a global minimizer of the problem
    \[
    \min_t\ \lambda(|t-a|^2 - \gamma^2a^2)_+ + \phi(t),
    \]
    whenever $\lambda \ge \frac{\alpha}{\gamma a}$, since $C$ can be chosen to be $\frac{\|A^\dagger\|}{\sigma} = \frac{1}{\gamma a}$ by Lemma~\ref{lem-1}. The bridge penalty function, on the other hand, is locally Lipschitz continuous everywhere except at $0$. Since $\gamma \in (0,1)$, $t^p$ is Lipschitz continuous on $[(1-\gamma)a/2,\infty)$ with Lipschitz constant $p[(1-\gamma)a/2]^{p-1}$. From Theorem~\ref{thm1} and Remark~\ref{remXX}, we conclude that any local minimizer of \eqref{ex:P0} is a local minimizer of the problem
    \[
    \min_t\ \lambda(|t-a|^2 - \gamma^2a^2)_+ + \phi(t),
    \]
    whenever $\lambda\ge \frac{p[(1-\gamma)a/2]^{p-1}}{\gamma a}$.
  \end{example}

  \subsection{When $\Phi$ is not locally Lipschitz continuous at some points}\label{sec42}

  In this subsection, we suppose that $\Phi(x)$ is not locally Lipschitz continuous at some points. 
To proceed, we make an assumption for $\Phi$ that will be used subsequently. 

  \begin{assumption}\label{assum1}
    The function $\Phi(x) = \sum_{i=1}^n\phi(x_i)$ is continuous and nonnegative with $\phi(0) = 0$, and is locally Lipschitz continuous everywhere except at $0$. Moreover,
    for any $L > 0$, there exists an $\epsilon > 0$ such that whenever $|t|< \epsilon$, we have
    \begin{equation}\label{phiineq}
    \phi(t) \ge L|t|.
    \end{equation}
  \end{assumption}
  It is not hard to show that the widely used bridge penalty function $|x|^p$, for $0 < p < 1$, satisfies this assumption.

  \begin{theorem}[{{\bf Local minimizers}}]\label{thm2}
    Suppose that $x^*$ is a local minimizer of \eqref{eq:P0} with a $\Phi$ satisfying Assumption~\ref{assum1}. Then there exists a $\lambda^* > 0$ such that $x^*$ is a local minimizer of \eqref{eq:P1} whenever $\lambda \ge \lambda^*$.
  \end{theorem}
  \begin{proof}
    Suppose first that $x^* = 0$. Fix any bounded neighborhood $U$ of $0$ and any $\lambda > 0$. Let $L$ denote a Lipschitz constant  the function $x\mapsto \lambda[(\|Ax - b\|^2 - \sigma^2)_+ + \|(Bx - h)_+\|_1]$ on $U$. For this $L$, by Assumption~\ref{assum1}, there exists a neighborhood $V\subseteq U$ of zero such that $\Phi(x)\ge L\|x\|_1$ whenever $x\in V$. Then for any $x\in V$, we have
    \[
    \begin{split}
    &\lambda[(\|Ax - b\|^2 - \sigma^2)_+ + \|(Bx - h)_+\|_1] + \Phi(x)\\
    & \ge \lambda[(\|Ax - b\|^2 - \sigma^2)_+ + \|(Bx - h)_+\|_1] + L\|x\|_1 \\
    & \ge \lambda[(\|Ax - b\|^2 - \sigma^2)_+ + \|(Bx - h)_+\|_1] + L\|x\|\\
    & \ge \lambda[(\|b\|^2 - \sigma^2)_+ + \|(- h)_+\|_1],
    \end{split}
    \]
    where the last inequality follows from the definition of $L$ being a Lipschitz constant. This shows that $x^* = 0$ is a local minimizer of \eqref{eq:P1} for any $\lambda > 0$.
    Thus, to complete the proof, from now on, we assume that $x^* \neq 0$. Let $I$ denote the support of $x^*$, i.e., $I := \{i:\; x_i^* \neq 0\}$. Then $I \neq \emptyset$.

    Since $x^*$ is a local minimizer of \eqref{eq:P0}, it follows that $x^*_I$ is a local minimizer of the following optimization problem:
    \begin{equation}\label{eq:P0I}
      \begin{array}{rl}
        \min\limits_{x_I} & \sum_{i\in I}\phi(x_i)\\
        {\rm s.t.} & \|A_I x_I - b\|\le \sigma,\ \ B_Ix_I \le h,\ \ D_I x_I \le d.
      \end{array}
    \end{equation}
    Let $\tilde\epsilon = \frac{1}{2}\min\{|x^*_i|:\;i\in I\} > 0$. Then we can choose a small $\delta > 0$ so that $x^*_I$ is a local minimizer of \eqref{eq:P0I} and that $\min_{i\in I}|x_i|>\tilde\epsilon$ for all $x_I \in {\bf B}(x^*_I;\delta)$. Next, consider $\cS_1 = {\bf B}(x^*_I;\delta)$ and $\cS_2 = \{x_I:\; \|A_I x_I - b\|\le \sigma,\ \ B_Ix_I \le h,\ \ D_I x_I \le d\}$. Then we have from \cite[Lemma~4.9]{LiNgPong07} that
    \[
    {\rm dist}(x_I,\cS_1\cap \cS_2)\le 4\ {\rm dist}(x_I,\cS_2)
    \]
    for all $x_I\in \cS_1$. Using this and applying Corollary~\ref{local-penalty} with $f(x_I) = \sum_{i\in I}\phi(x_i)$, the $\cS_1$ and $\cS_2$ as defined above, and
    \[
    Q(x_I) = 4C\left[(\|A_Ix_I - b\|^2 - \sigma^2)_+ + \|(B_Ix_I - h)_+\|_1 + \|(D_I x_I - d)_+\|_1\right]
     \]
    for the $C$ given in Lemma~\ref{lem0}, we conclude that there exists a $\lambda^* > 0$ so that for any $\lambda \ge \lambda^*$, there is a neighborhood $U_I$ of $0$ with $U_I\subseteq {\bf B}(0;\frac{\delta}{2})$ such that $G^I_\lambda(x_I)\ge G^I_\lambda(x_I^*)$ whenever $x_I\in x_I^* + U_I$,
    where
    \[
    G^I_\lambda(x_I) = \lambda \left[(\|A_Ix_I - b\|^2 - \sigma^2)_+ + \|(B_Ix_I - h)_+\|_1 + \|(D_I x_I - d)_+\|_1\right] + \sum_{i\in I}\phi(x_i).
    \]

    We now show that $x^*$ is a local minimizer of \eqref{eq:P1} with $\lambda \ge \lambda^*$. Fix any $\epsilon > 0$ and any $\lambda \ge \lambda^*$. Consider the (bounded) neighborhood $U:= U_I\times (-\epsilon,\epsilon)^{n-|I|}$ of $0$ and let $M$ be a Lipschitz constant  the function
    \[
    g_\lambda(x) = \lambda \left[(\|Ax - b\|^2 - \sigma^2)_+ + \|(Bx - h)_+\|_1 + \|(D x - d)_+\|_1\right]
    \]
    over $x^* + U$. Taking $L = M$ in Assumption~\ref{assum1}, we see that there exists an $\epsilon_0 \in(0, \epsilon)$ such that \eqref{phiineq} holds with $M$ in place of $L$ whenever $|t|<\epsilon_0$. Then, for any $v\in U_I\times (-\epsilon_0,\epsilon_0)^{n-|I|}$ with $x^* + v \in S_1$, we have
    \begin{equation*}
      \begin{split}
        & F_{\lambda}(x^* + v) = F_{\lambda}\left(x^* + \begin{pmatrix}
          v_I\\
          v_{\bar I}
        \end{pmatrix}\right) = g_\lambda\left(x^* + \begin{pmatrix}
          v_I\\
          v_{\bar I}
        \end{pmatrix}\right) + \sum_{i\in I}\phi(x^*_i + v_i) + \sum_{i\notin I}\phi(v_i)\\
        & \ge g_\lambda \begin{pmatrix}
          x^*_I + v_I\\ 0
        \end{pmatrix} - M\|v_{\bar I}\| + \sum_{i\in I}\phi(x^*_i + v_i) + M\|v_{\bar I}\|_1\\
        & \ge G^I_\lambda(x_I^*) = F_{\lambda}(x^*),
      \end{split}
    \end{equation*}
    where the first inequality follows from the Lipschitz continuity of $g_\lambda$ with Lipschitz constant $M$ and \eqref{phiineq} with $L = M$, and the last inequality follows from the local optimality of $x_I^*$, while the second and the last equalities follow from $\|(D(x^* + v) - d)_+\|_1 = 0$ since $x^* + v\in S_1$. This shows that $x^*$ is locally optimal for \eqref{eq:P1} with $\lambda \ge \lambda^*$, and completes the proof.
  \end{proof}


  We next study $\epsilon$-minimizers of \eqref{eq:P0} and \eqref{eq:P1}, which are defined as follows.
  \begin{definition}
Let $\epsilon > 0$.
\begin{enumerate}
\item We say that $x_\epsilon$ is an $\epsilon$-minimizer of \eqref{eq:P0},  if $x_\epsilon\in S$ and $\Phi(x_\epsilon)\le \inf\limits_{x\in S} \Phi(x) + \epsilon$.
\item
 We say that $x_\epsilon$ is an $\epsilon$-minimizer of \eqref{eq:P1}, if $x_\epsilon\in S_1$ and
 $F_\lambda(x_\epsilon)\le \inf\limits_{x\in S_1} F_\lambda(x) + \epsilon$.
 \end{enumerate}
\end{definition}

 In order to establish results concerning $\epsilon$-minimizers, we also need the following definition.
  \begin{definition}
    We say that a globally Lipschitz continuous function $\Psi$ with a Lipschitz constant  $L$ is an $(L,\epsilon)$-approximation to $\Phi$ if
    $0\le \Psi(x) - \Phi(x)\le \epsilon$ for all $x$.
  \end{definition}

%

As a concrete example of such an approximation, consider the case where $\Phi(x) = \sum_{i=1}^n\phi(x_i)$ with $\phi(t) = |t|^p$ for some $0<p<1$.
We can consider the following smoothing function of $|t|$:
\[
\psi_\mu(t)=\begin{cases}
  |t| \,\quad &
{\rm if} \, |t|\ge \mu,\\
\frac{t^2}{2\mu}+\frac{\mu}{2} \,\quad &  {\rm otherwise}.
\end{cases}
\]
 Notice that for a fixed $\mu>0$, the minimum and maximum values of $\psi_\mu(t)-|t|$ are attained at $|t|\ge\mu$ and $t=0$, respectively. Let
 \begin{equation*}
 \Psi_\mu(x)=\sum^n_{i=1}\psi_\mu(x_i)^p.
 \end{equation*}
 Then we have from the above discussion and Lemma~\ref{lem2} that
 \begin{equation}\label{s2}
 0\le \Psi_\mu(x)-\|x\|^p_p\le n\left(\frac{\mu}{2}\right)^p.
 \end{equation}
 Moreover, for a fixed $\mu>0$, the function $\Psi_\mu$ is continuously
differentiable. The maximum value of $|(\psi_\mu(t)^p)'|$ is attained at $t=\mu$, and hence we have
\begin{equation}\label{s3}
 |\Psi_\mu(x)-\Psi_\mu(y)|\le \sqrt{n}p\mu^{p-1}\|x-y\|.
 \end{equation}
The inequalities \eqref{s2} and \eqref{s3} show that $\Psi_\mu$ is a $(\sqrt{n}p\mu^{p-1},n(\mu/2)^p)$-approximation to $\Phi$
when $\phi(t) = |t|^p$.


  From the definition of an $(L,\epsilon)$-approximation $\Psi$,
  it is easy to show that any global minimizer of
  \begin{equation}\label{eq:Pu}
    \begin{array}{rl}
      \min\limits_{x\in S_1} & \Psi(x) \\
      {\rm s.t.} & \|A x - b\| \le \sigma,\ \ Bx \le h,
    \end{array}
  \end{equation}
 is an $\epsilon$-minimizer of \eqref{eq:P0}. Conversely, any global minimizer $x^*$ of \eqref{eq:P0} is an $\epsilon$-minimizer of \eqref{eq:Pu}.
  Our next result concerns the global minimizers of \eqref{eq:P0} and the $\epsilon$-minimizers of \eqref{eq:P1}.

  \begin{theorem}[{{\bf $\epsilon$-minimizers}}]
    Suppose that $\Phi$ admits an $(L,\epsilon/2)$-approximation $\Psi$.
    Then for any global minimizer $x^*$ of \eqref{eq:P0}, there exists a $\lambda^* > 0$ so that $x^*$ is an $\epsilon$-minimizer of \eqref{eq:P1} whenever $\lambda \ge \lambda^*$, i.e.,
    \begin{equation}\label{eq:rel1}
    F_\lambda(x^*) \le \inf_{x\in S_1} F_\lambda(x) + \epsilon;
    \end{equation}
    in particular, one can take $\lambda^* = CL$, where $C$ is the constant in Lemma~\ref{lem0}.
  \end{theorem}
  \begin{proof}
     From the definition of an $(L,\epsilon/2)$-approximation, we see that any global minimizer $x^*$ of \eqref{eq:P0} is an $\epsilon/2$-minimizer of \eqref{eq:Pu}. Moreover, since $\Psi$ is globally Lipschitz continuous with Lipschitz constant $L$, we have for any $x\in S_1$ that
    \begin{equation*}
      \begin{split}
       \tilde L\ {\rm dist}(x,S) + \Psi(x) = \tilde L\ \|x - P_S(x)\| + \Psi(x) \ge \Psi(P_S(x)) \ge \Psi(x^*) - \frac\epsilon 2,
      \end{split}
    \end{equation*}
    where $\tilde L$ is any number greater than or equal to $L$ and the second inequality follows from the $\epsilon/2$-optimality of $x^*$ for \eqref{eq:Pu}. This shows that $x^*$ is an $\epsilon/2$-minimizer of
    the optimization problem
    \begin{equation*}
      \min_{x\in S_1} \ \tilde L\ {\rm dist}(x,S) + \Psi(x).
    \end{equation*}
    Combining this fact with Lemma~\ref{lem0}, it is not hard to show that $x^*$ is an $\epsilon/2$-minimizer of
    \begin{equation*}
        \min\limits_{x\in S_1}\ C\tilde L \left[(\|Ax - b\|^2 - \sigma^2)_+ + \|(Bx - h)_+\|_1\right] + \Psi(x).
    \end{equation*}
    Using this and the fact that $0\le \Psi(x) - \Phi(x) \le \epsilon/2$ for all $x$, we have further that for all $x\in S_1$,
    \begin{equation*}
      \begin{split}
        & F_{C\tilde L}(x) = C\tilde L\left[(\|Ax - b\|^2 - \sigma^2)_+ + \|(Bx - h)_+\|_1\right] + \Phi(x)\\
        & \ge C\tilde L\left[(\|Ax - b\|^2 - \sigma^2)_+ + \|(Bx - h)_+\|_1\right] + \Psi(x) - \frac\epsilon 2\\
        & \ge C\tilde L\left[(\|Ax^* - b\|^2 - \sigma^2)_+ + \|(Bx^* - h)_+\|_1\right] + \Psi(x^*) - \frac\epsilon 2 - \frac\epsilon 2\\
        & = F_{C\tilde L}(x^*) - \epsilon,
      \end{split}
    \end{equation*}
    i.e., \eqref{eq:rel1} holds with $\lambda^* = CL$.
  \end{proof}

  So far we have shown that if $x^*$ is locally or globally optimal for \eqref{eq:P0}, then it is also optimal in some sense for \eqref{eq:P1}, when $\lambda$ is sufficiently large. Conversely, it is clear that if $x^*$ is optimal (locally or being an $\epsilon$-minimizer) for \eqref{eq:P1} for some $\lambda>0$, and $x^*$ is also feasible for \eqref{eq:P0}, then it is also optimal for \eqref{eq:P0}. Our next result studies the case when $x^*$ is not necessarily feasible for \eqref{eq:P0}.

  \begin{theorem}[{{\bf $\epsilon$-minimizers feasible for \eqref{eq:P0}}}]\label{thm3.5}
    Suppose that $\Phi(x) = \sum_{i=1}^n\phi(x_i)$ with $\phi$ being H\"{o}lder continuous for some $0<p<1$, i.e., there exists a $K > 0$ such that
    \[
    |\phi(s) - \phi(t)|\le K|s - t|^p
    \]
    for any $s$, $t\in \R$. Take any $\epsilon > 0$ and fix any $\tilde x\in S$. Consider any
    \[
    \lambda > \frac{K^\frac{1}{p}C\Phi(\tilde x)}{(n^{\frac{p}{2}-1}\epsilon)^\frac1p},
    \]
    with $C$ chosen as in Lemma~\ref{lem0}.
    Then for any global minimizer $x_\lambda$ of \eqref{eq:P1}, the projection $P_S(x_\lambda)$ is an $\epsilon$-minimizer of \eqref{eq:P0}.
  \end{theorem}
  \begin{proof}
  We first note from the global optimality of $x_\lambda$ that $F_\lambda(x_\lambda)\le F_\lambda(\tilde x)$, from which we immediately obtain that
  \begin{equation}\label{key1}
    (\|Ax_\lambda - b\|^2 - \sigma^2)_+ + \|(Bx_\lambda - h)_+\|_1 \le \frac{1}{\lambda} F_\lambda (x_\lambda)\le \frac{1}{\lambda} F_\lambda (\tilde x) = \frac 1\lambda \Phi(\tilde x).
  \end{equation}
  Next, for the projection $P_S(x_\lambda)$, we have
  \begin{equation}\label{key2}
    \begin{split}
      &\Phi(P_S(x_\lambda)) - \Phi(x_\lambda)
      \le K\sum_{i=1}^n \left|[P_S(x_\lambda)]_i - [x_\lambda]_i\right|^p = nK\frac1n\sum_{i=1}^n \left(\left|[P_S(x_\lambda)]_i - [x_\lambda]_i\right|^2\right)^\frac p2\\
      & \le nK\left(\frac1n\sum_{i=1}^n\left|[P_S(x_\lambda)]_i - [x_\lambda]_i\right|^2\right)^\frac p2 = Kn^{1 -\frac{p}{2}} \|P_S(x_\lambda) - x_\lambda\|^p\\[5 pt]
      & = Kn^{1 -\frac{p}{2}} {\rm dist}^p(x_\lambda,S)\le KC^pn^{1 -\frac{p}{2}}\left[(\|Ax_\lambda - b\|^2 - \sigma^2)_+ + \|(Bx_\lambda - h)_+\|_1\right]^p\\[5 pt]
      & \le Kn^{1 -\frac{p}{2}}\left(\frac{C\Phi(\tilde x)}{\lambda}\right)^p,
    \end{split}
  \end{equation}
  where the first inequality follows from the assumption on H\"{o}lder continuity, the second one holds due to the concavity of the function $t\mapsto t^\frac{p}2$ for nonnegative $t$, the third inequality follows from Lemma~\ref{lem0} and the fact that $x_\lambda\in S_1$, while the last one follows from \eqref{key1}. On the other hand, for any $x\in S$, we have from the optimality of $x_\lambda$ for \eqref{eq:P1} and the definition of $F_\lambda$ that $F_\lambda(x_\lambda) \le F_\lambda (x) = \Phi(x)$. From this we see immediately that
  \begin{equation*}
    \Phi(x_\lambda)\le F_\lambda(x_\lambda)\le  \inf_{x\in S}\Phi(x).
  \end{equation*}
  Combining this with \eqref{key2}, we obtain further that
  \[
  0\le \Phi(P_S(x_\lambda)) - \inf_{x\in S}\Phi(x) \le Kn^{1 -\frac{p}{2}}\left(\frac{C\Phi(\tilde x)}{\lambda}\right)^p < \epsilon,
  \]
  from our choice of $\lambda$. This shows that $P_S(x_\lambda)$ is an $\epsilon$-minimizer of \eqref{eq:P0}.
  \end{proof}

  From Lemma~\ref{lem2}, it is easy to see that $t\mapsto |t|^p$, $0<p<1$, is H\"{o}lder continuous with $K = 1$. Thus, we have the following immediate corollary when $\Phi(x) = \|x\|_p^p$, $0<p<1$.
  \begin{corollary}
    Suppose that $\Phi(x) = \|x\|_p^p$ for some $0<p<1$. Take any $\epsilon > 0$ and fix any $\tilde x\in S$. Consider any
    \[
    \lambda > \frac{C\|\tilde x\|_p^p}{(n^{\frac{p}{2}-1}\epsilon)^\frac1p},
    \]
    with $C$ chosen as in Lemma~\ref{lem0}.
    Then for any global minimizer $x_\lambda$ of \eqref{eq:P1}, the projection $P_S(x_\lambda)$ is an $\epsilon$-minimizer of \eqref{eq:P0}.
  \end{corollary}

  \section{Algorithm}\label{sec5}

  In this section we propose a penalty method for solving problem \eqref{eq:P0}. Based on our discussion in the previous section, a natural penalty method for solving \eqref{eq:P0} would be to solve
the problem \eqref{eq:P1} once with an exact penalty parameter $\lambda$. This approach is, however, not appealing in practice because such $\lambda$ may be hard to estimate, or it may be 
over-estimated and the resulting penalty problem becomes very ill-conditioned. To circumvent these potential difficulties, we propose a practical penalty method that solves a sequence of penalty subproblems in the form of \eqref{eq:P1} with a gradually increased penalty parameter. In addition, the approximate solution of the current subproblem will be used as the starting point for solving the next subproblem.

  Our algorithm is presented in Section~\ref{sec:subalg}, where we show that any cluster point of the sequence generated from our algorithm is a KKT point of problem~\eqref{eq:P0}, under a suitable constraint qualification. To prepare for our convergence analysis, we start by discussing the first-order optimality conditions for problems \eqref{eq:P0} and \eqref{eq:P1} and describing the constraint qualification in Section~\ref{sec3}.

  \subsection{First-order optimality conditions}\label{sec3}

In this subsection, we discuss the first-order optimality conditions for problems \eqref{eq:P0} and \eqref{eq:P1}.

  We first look at model \eqref{eq:P1}. Since the objective is a sum of a locally Lipschitz continuous function and the lower semicontinuous function $\Phi+\delta_{S_1}$, it follows from \cite[Theorem~8.15]{Rock98}, \cite[Theorem~10.1]{Rock98} and \cite[Exercise~10.10]{Rock98} that at any locally optimal solution $\bar x$ of \eqref{eq:P1}, we have
  \begin{equation}\label{P11stoptimality}
    0 \in \partial (\lambda (\|A\cdot - b\|^2 - \sigma^2)_+)(\bar x) + \partial (\lambda\|(B\cdot - h)_+\|_1)(\bar x) + \partial (\Phi + \delta_{S_1})(\bar x).
  \end{equation}
  This motivates the following definition.

  \begin{definition}[{{\bf First-order stationary point of \eqref{eq:P1}}}]
    We say that $x^*$ is a first-order stationary point of \eqref{eq:P1} if $x^* \in S_1$ and \eqref{P11stoptimality} is satisfied with $x^*$ in place of $\bar x$.
  \end{definition}

  In the special case where $\Phi(x) = \sum_{i=1}^n\phi(x_i)$ with $\phi(t) = |t|^p$, it is easy to check that $\partial \phi(t) = \{p\,{\rm sign}(t)\,|t|^{p-1}\}$ whenever $t\neq 0$ and, from Lemma~\ref{lem:lp} (i), we have $\partial \phi(0) = \R$. Moreover, for the first subdifferential in \eqref{P11stoptimality}, we have the following explicit expression
  \begin{equation} \label{sub-diff}
    \partial (\lambda (\|A\cdot - b\|^2 - \sigma^2)_+)(\bar x) =
    \begin{cases}
      0 & {\rm if\ } \|A\bar x - b\| < \sigma,\\
      {\rm conv}\{0,2\lambda A^T(A\bar x - b)\} &  {\rm if\ }\|A\bar x - b\| = \sigma,\\
      2\lambda A^T(A\bar x - b) & {\rm otherwise. \ }
    \end{cases}
  \end{equation}
  Thus, in the case when $B$ is vacuous and $S_1 = \R^n$, we have that $x^*$ is a first-order stationary point of \eqref{eq:P1} if and only if
  \begin{equation}\label{Lp1stordereach}
  0 = 2\nu\lambda [A^T(A x^*- b)]_i + p\,{\rm sign}(x^*_i)\,|x^*_i|^{p-1},\ \ \forall i\in I
  \end{equation}
  with $I = \{i:\; x^*_i \neq 0\}$ for some $\nu$ satisfying
  \[
  \nu
    \begin{cases}
      = 0 & {\rm if\ } \|Ax^*- b\| < \sigma,\\
      \in [0,1] &  {\rm if\ }\|Ax^* - b\| = \sigma,\\
      = 1 & {\rm otherwise. \ }
    \end{cases}
  \]
  This is because the inclusion \eqref{P11stoptimality} is trivial for $i\notin I$. Using the definition of $I$, it is not hard to see that \eqref{Lp1stordereach} is further equivalent to
  \begin{equation}\label{Lp1storderall}
    0 = 2\nu \lambda{\rm Diag}(x^*)A^T(A x^*- b) + p |x^*|^{p},
  \end{equation}
  with the same $\nu$ defined above.

  We next turn to the KKT points of \eqref{eq:P0}. We recall from \cite[Theorem~8.15]{Rock98} that at any locally optimal solution $\bar x$ of \eqref{eq:P0}, we have
  \begin{equation}\label{P01stoptimality}
    0 \in \cN_{S_2}(\bar x) + \partial (\Phi + \delta_{S_1})(\bar x),
  \end{equation}
  assuming the following constraint qualification holds:
  \begin{equation}\label{CQ0}
    -\partial^{^\infty}\!(\Phi + \delta_{S_1})(\bar x) \cap \cN_{S_2}(\bar x) = \{0\}.
  \end{equation}
  This motivates the following definition.
  \begin{definition}[{{\bf KKT point of \eqref{eq:P0}}}]
    We say that $x^*$ is a KKT point of \eqref{eq:P0} if $x^*\in S$ and \eqref{P01stoptimality} is satisfied with $x^*$ in place of $\bar x$.
  \end{definition}
  Since there exists $x_0$ with $\|Ax_0-b\| < \sigma$, in the case when $B$ is vacuous and $S_1 = \R^n$, we have
  \begin{equation}\label{eq:normalS}
    \cN_S(\bar x) =\begin{cases}
       \{\mu A^T(A\bar x - b):\; \mu \ge 0\}\neq \{0\}& {\rm if}\ \|A\bar x - b\| = \sigma,\\
       \{0\}& {\rm if}\ \|A\bar x - b\| < \sigma;
    \end{cases}
  \end{equation}
  see, for example, Theorem~1.3.5 in \cite[Section~D]{HiLe01}.
  In the special case where $\Phi(x) = \sum_{i=1}^n\phi(x_i)$ with $\phi(t) = |t|^p$ and that $B$ is vacuous and $S_1 = \R^n$, similarly as above, one can see that an $x^*$ satisfying $\|Ax^*-b\| = \sigma$ is a KKT point of \eqref{eq:P0} if and only if there exists a $\mu \ge 0$ so that
  \[
  0 = \mu [A^T(A x^*- b)]_i + p\,{\rm sign}(x^*_i)\,|x^*_i|^{p-1},\ \ \forall i\in I,
  \]
  with $I = \{i:\; x^*_i \neq 0\}$. This condition is further equivalent to
  \begin{equation}\label{Lp1storder}
    0 = \mu {\rm Diag}(x^*)A^T(A x^*- b) + p |x^*|^{p}.
  \end{equation}
  On the other hand, we recall from Lemma~\ref{lem:lp} (ii) that
  \[
  \partial^{^\infty}\!\!\Phi(x^*) = \{v:\; v_i = 0 {\rm \ for\ }i \in I\}.
  \]
  Since $\cN_S(x^*) = \{\mu A^T(Ax^* - b):\; \mu \ge 0\}$, the constraint qualification \eqref{CQ0} is equivalent to
  $[A^T(Ax^* - b)]_i$ being nonzero for some $i\in I$. From the definition of $I$, this constraint qualification can be equivalently formulated as
  \begin{equation}\label{CQ}
    {\rm Diag}(x^*)A^T(Ax^* - b) \neq 0.
  \end{equation}

On passing, recall from Proposition~5.3.1~(i) and Remark~5.3.2 in \cite[Section~A]{HiLe01} that we have
\[
\cN_{S_2}(x) = \cN_{\|A\cdot - b\|\le \sigma}(x) + \cN_{B\cdot \le h}(x)
\]
at any $x\in S_2$, thanks to the existence of $x_0\in S$ with $\|Ax_0 - b\|< \sigma$ by our blanket assumption. It is then not hard to see from the definitions that any first-order stationary point of \eqref{eq:P1} that lies in $S$ is a KKT point of \eqref{eq:P0}. Conversely,
any KKT point of \eqref{eq:P0} is a first-order stationary point of \eqref{eq:P1} for some $\lambda > 0$.

%
%
%

Before ending this subsection, we comment on the magnitude of the nonzero entries of a first-order stationary point $x^*$ of \eqref{eq:P1}, assuming $\Phi(x) = \sum_{i=1}^n\phi(x_i)$ for some continuous function $\phi$. To facilitate comparison with existing work, we focus on the case where $B$ is vacuous and $S_1 = \R^n$. Note that in this case, the definition of $F_{\lambda}(x)$ reduces to $\lambda (\|Ax - b\|^2 - \sigma^2)_+ + \Phi(x)$. Then it follows from the stationarity of $x^*$ and \eqref{P11stoptimality} that there exists $0\le \nu\le 1$ so that at any $i$ with $x_i^*\neq 0$, we have for some $\xi_i\in \partial \phi(x_i^*)$,
  \[
  -\xi_i = 2\nu\lambda[A^T(Ax^* - b)]_i.
  \]
  Let $x^\diamond$ be chosen so that $F_{\lambda}(x^*)\le F_{\lambda}(x^\diamond)$. Then for each $i$ with $x_i \neq 0$,
  \begin{equation}\label{upperbound}
    \begin{split}
      &|\xi_i| \le 2\lambda\|A^T(Ax^* - b)\| \le 2\lambda \|A\|\|Ax^* - b\| \\
      & \le 2\sqrt{\lambda}\|A\|\sqrt{(\lambda \|Ax^* - b\|^2 - \lambda \sigma^2)_+ + \lambda \sigma^2}\\
      & \le 2\sqrt{\lambda}\|A\|\sqrt{F_\lambda(x^*) + \lambda \sigma^2} \le 2\sqrt{\lambda}\|A\|\sqrt{F_\lambda(x^\diamond) + \lambda\sigma^2},
    \end{split}
  \end{equation}
  where the fourth inequality follows from the nonnegativity of $\Phi$, and the last inequality follows from the choice of $x^\diamond$. A concrete lower bound can be derived for some specific $\phi$. For example, consider $\phi(t) = |t|^p$ for $p\in (0,1)$. Then we have from \eqref{upperbound} that for $x_i^*\neq 0$,
  \begin{equation}\label{lowerbd}
  p|x_i^*|^{p-1}\le 2\sqrt{\lambda}\|A\|\sqrt{F_\lambda(x^\diamond) + \lambda\sigma^2} \Longrightarrow |x_i^*|\ge \left(\frac{p}{2\sqrt{\lambda}\|A\|\sqrt{F_\lambda(x^\diamond) + \lambda\sigma^2}}\right)^{\frac1{1-p}} > 0.
  \end{equation}
  Since local minimizers of \eqref{eq:P0} are local minimizers of \eqref{eq:P1} for some $\lambda^* > 0$ according to Theorem~\ref{thm2}, and local minimizers of \eqref{eq:P1} are first-order stationary points of \eqref{eq:P1}, the above discussion also gives a lower bound on the magnitude of the nonzero entries of the local minimizers of \eqref{eq:P0} when $B$ is vacuous and $S_1 = \R^n$.

  \begin{remark}
    In the recent paper \cite{CXY}, the authors derived a lower bound on the magnitudes of the nonzero entries of any first-order stationary point $\hat x$ of \eqref{eq:P-1} with $H_\lambda(x) = \lambda \|Ax - b\|^2 + \|x\|_p^p$ for some $0<p<1$. Their lower bound is given by
    \[
    |\hat x_i| \ge \left(\frac{p}{2\sqrt{\lambda}\|A\|\sqrt{H_\lambda(\tilde x)}}\right)^\frac{1}{1-p} > 0,\ {\it for\ } \hat x_i \neq 0,
    \]
    with $\tilde x$ chosen so that $H_\lambda(\hat x)\le H_\lambda(\tilde x)$; see \cite[Theorem~2.3]{CXY}. This lower bound is similar to \eqref{lowerbd} except that $F_\lambda(x^\diamond) + \lambda\sigma^2$ is replaced by $H_\lambda(\tilde x)$. Notice that when $x^\diamond = \tilde x$, we always have $F_\lambda(x^\diamond) + \lambda\sigma^2 \ge H_\lambda(x^\diamond)$, and these two values are the same if $\|Ax^\diamond - b\| \ge \sigma$. In particular, when $x^\diamond = \tilde x = 0$ and $\|b\|\ge \sigma$, the guaranteed lower bounds for both models are the same and is given by $\big(\frac{p}{2\lambda\|A\|\|b\|}\big)^{\frac1{1-p}}$.
  \end{remark}

  \subsection{Penalty method for solving \eqref{eq:P0}}\label{sec:subalg}

In this subsection, we present details of our penalty method for solving \eqref{eq:P0}. Before proceeding, we make the following assumption on $\Phi$ and $S_1$, which is standard in 
guaranteeing the sequence generated by an algorithm is bounded.
\begin{assumption}\label{assum2}
  The function $\Phi + \delta_{S_1}$ has bounded level sets.
\end{assumption}

Based on our previous discussions, an $\epsilon$-minimizer of \eqref{eq:P0} can be obtained by finding a globally optimal solution of \eqref{eq:P1} with a sufficiently large $\lambda$.  This approach is,  however,  not appealing because such $\lambda$ may be hard to estimate, or it may be over-estimated  and the resulting penalty problem becomes very ill-conditioned.
Instead, it is natural to solve a sequence of problems in the form of \eqref{eq:P1} in which $\lambda$ gradually
increases. This scheme is commonly used in the classical penalty method. Also, notice that the first part of the
objective of \eqref{eq:P1} is convex but nonsmooth. For an efficient implementation, we solve a sequence of partially smooth
counterparts of \eqref{eq:P1} in the form of
\begin{equation}\label{eq:P1-mu}
 \min_{x\in S_1} \ F_{\lambda,\mu}(x) := f_{\lambda,\mu}(x) + \Phi(x)
\end{equation}
for some $\lambda, \mu >0$,
where
\[
\begin{split}
f_{\lambda,\mu}(x) := h_{\lambda,\mu}(\|Ax - b\|^2 - \sigma^2) + \sum_{i=1}^\ell h_{\lambda,\mu}([Bx - h]_i) \ {\rm with\ \ }
h_{\lambda,\mu}(s) := \lambda\max\limits_{0 \le t \le 1}  \left\{st - \frac{\mu}{2} t^2\right\},
\end{split}
\]
where the function $h_{\lambda,\mu}(\cdot)$ is a $\mu$-smoothing for the function $s\to \lambda\cdot s_+$; see \cite[Eq.~4]{PinarZe94} and the discussions therein.

It is not hard to show that for all $x \in \R^n$,
\begin{equation}\label{f-prop}
0 \ \le \ f_{\lambda,\mu}(x) \ \le \ \lambda [(\|Ax - b\|^2 - \sigma^2)_+ +  \|(Bx - h)_+\|_1] \ \le \ f_{\lambda,\mu}(x) + \frac{\ell + 1}2\lambda \mu,
\end{equation}
and
\begin{equation} \label{f-deriv}
  \nabla f_{\lambda,\mu}(x) = 2 h'_{\lambda,\mu}(\|Ax - b\|^2 - \sigma^2) A^T(Ax-b) + \sum_{i=1}^\ell h'_{\lambda,\mu}([Bx - h]_i)b_i,
\end{equation}
where $b_i$ is the column vector formed from the $i$th row of $B$, and the function $h'_{\lambda,\mu}$ satisfies
\begin{eqnarray}
& & h'_{\lambda,\mu}(s) = \lambda\min\left\{\max\left\{\frac{s}{\mu},0\right\},1\right\}, \label{h-deriv} \\ [5pt]
& & |h'_{\lambda,\mu}(s_1) - h'_{\lambda,\mu}(s_2)| \le \frac{\lambda}{\mu} |s_1-s_2| \quad\quad \forall s_1, \ s_2 \in \R.
\label{h-prop} 
\end{eqnarray}
To solve \eqref{eq:P1-mu}, we consider an adaptation of the nonmonotone proximal gradient (NPG) method proposed in
\cite{WrNoFi09}. In \cite{WrNoFi09}, the NPG method was proposed to solve a class of unconstrained problems in the form of
\beq \label{fP}
\min\limits_x f(x) + P(x),
\eeq
where $f$ and $P$ are finite-valued functions in $\R^n$, and moreover, $f$ is differentiable in $\R^n$ and its gradient
is globally Lipschitz continuous in $\R^n$. The convergence analysis for the NPG method conducted in \cite{WrNoFi09} relies
on the global Lipschitz continuity of $\nabla f$.  Though the objective of \eqref{eq:P1-mu} is in the same form as that of
\eqref{fP},  we observe from \eqref{f-deriv} that $\nabla f_{\lambda,\mu}$ is locally but not globally Lipschitz continuous
in $\R^n$.
It thus appears that the NPG method \cite{WrNoFi09} may not be applicable to our problem \eqref{eq:P1-mu}. We are, however, fortunately
able to show in Appendix~\ref{sec:NPG} that this NPG method is indeed capable of solving a more general class of problems that satisfies
Assumption~\ref{assump-fphi}.  We next verify that Assumption~\ref{assump-fphi} holds for problem \eqref{eq:P1-mu} with
$f=f_{\lambda,\mu}$ and $P=\Phi + \delta_{S_1}$. As a consequence, the NPG method is applicable to our problem \eqref{eq:P1-mu}.

First, it is easy to see that Assumption \ref{assump-fphi} (ii) holds.
Let $x^0 \in S_1$ be arbitrarily chosen. It follows from \eqref{f-prop} that
$f_{\lambda,\mu}(x) \ge 0$, which implies that
\begin{equation}\label{bound}
\Omega(x^0):= \left\{x\in S_1: F_{\lambda,\mu}(x)\le F_{\lambda,\mu}(x^0)\right\} \subseteq
\left\{x\in S_1: \Phi(x) \le F_{\lambda,\mu}(x^0)\right\}.
\end{equation}
The set on the right hand side is nonempty and bounded by Assumption~\ref{assum2}, and hence $\Omega(x^0)$ is nonempty and compact. Since $f_{\lambda,\mu} + \Phi$ is a continuous function, it is uniformly continuous and bounded below in $\Omega(x^0)$.
Consequently, Assumption \ref{assump-fphi} (iii) holds. One can also easily verify that Assumption \ref{assump-fphi} (iv) holds using the compactness of $\Omega(x^0)$ and the nonnegativity of $\Phi$.
Finally, it is routine to show that $\nabla f_{\lambda,\mu}$ is locally Lipschitz continuous. This together with the compactness of $\Omega(x^0)$ shows that Assumption~\ref{assump-fphi} (i) also holds.
Therefore, the NPG method can be suitably applied to solving problem \eqref{eq:P1-mu}.

We now establish a convergence result for the NPG method applied to problem \eqref{eq:P1-mu}.

\begin{theorem} \label{converge-P1-mu}
Suppose that Assumption~\ref{assum2} holds. Given any $x^0 \in S_1$, let $\{x^k\}$ be the sequence generated by the NPG method applied to
problem \eqref{eq:P1-mu}. There hold:
\bi
\item[(i)] $\{x^k\}$ is bounded;
\item[(ii)] Any accumulation point $x^*$ of $\{x^k\}$ is a first-order stationary point of problem \eqref{eq:P1-mu},
that is, it satisfies
\beq \label{stat-pt}
0\in \nabla f_{\lambda,\mu}(x^*) + \partial (\Phi + \delta_{S_1})(x^*).
\eeq
\ei
\end{theorem}

\begin{proof}
(i) It follows from \eqref{bound} and Proposition \ref{inner-convergence} (i) with
$f=f_{\lambda,\mu}$ and $P=\Phi + \delta_{S_1}$ that
\[
\{x^k\} \subseteq \{x\in S_1:\; F_{\lambda,\mu}(x) \le F_{\lambda,\mu}(x^0)\} \subseteq \{x\in S_1:\; \Phi(x) \le F_{\lambda,\mu}(x^0)\}
\]
and hence $\{x^k\}$ is bounded.

(ii) In view of Proposition
\ref{inner-convergence} (ii), $\bar L_k \le \tilde L$ for some $\tilde L>0$ and all $k \ge 0$. It follows from
\eqref{prox-subprob} with $f=f_{\lambda,\mu}$ and $P=\Phi + \delta_{S_1}$, together with \cite[Theorem~10.1]{Rock98} and \cite[Exercise~10.10]{Rock98} that
\[
0\in \nabla f_{\lambda,\mu}(x^k) + \bar L_k (x^{k+1}-x^k) + \partial (\Phi + \delta_{S_1})(x^{k+1}).
\]
Suppose that $x^*$ is an accumulation point of $\{x^k\}$. Then there exists a subsequence $\cK$ such that
$\{x^k\}_{\cK} \to x^*$. Upon taking limits as $k\in \cK \to \infty$ on both sides of the above inclusion and
using Theorem \ref{outer-convergence} and \eqref{outersemi}, we see that \eqref{stat-pt} holds.
\end{proof}

We are now ready to present a penalty method for solving problem \eqref{eq:P0}.

\gap

\noindent
{\bf Penalty method for problem \eqref{eq:P0}:} \\ [5pt]
 Let $\xfeas$ be an arbitrary feasible point of problem \eqref{eq:P0}. Choose $x^0\in S_1$,
$\lambda_0>0$, $\mu_0>0$, $\epsilon_0>0$, $\rho>1$ and $\theta \in (0,1)$ arbitrarily.
Set $k=0$ and $x^{0,0}=x^0 \in S_1$.
\begin{itemize}
\item[1)] If $F_{\lambda_k,\mu_k}(x^{k,0}) > F_{\lambda_k,\mu_k}(\xfeas)$, set $x^{k,0}=\xfeas$. Apply the NPG
method with $x^{k,0}$ as the initial point to find an approximate stationary point $x^k$ to problem \eqref{eq:P1-mu} with
$\lambda = \lambda_k$ and $\mu=\mu_k$ satisfying
\beq \label{inner-cond}
{\rm dist}(0,\nabla f_{\lambda_k,\mu_k}(x^k) + \partial (\Phi + \delta_{S_1})(x^k)) \le \ \epsilon_k.
\eeq
\item[2)]
Set $\lambda_{k+1} = \rho \lambda_k$, $\mu_{k+1} = \theta \mu_k$, $\eps_{k+1}=\theta \eps_k$ and
$x^{k+1,0} = x^k$.
\item[3)]
Set $k \leftarrow k+1$ and go to step 1).
\end{itemize}
\noindent
{\bf end}

\vgap

\begin{remark}
By virtue of Theorem \ref{converge-P1-mu}, an $x^k$ satisfying \eqref{inner-cond} can be found by the NPG
method within a finite number of iterations. Therefore, the sequence $\{x^k\}$ is well defined.
\end{remark}

\gap

Convergence results for the above penalty method for solving problem \eqref{eq:P0} are presented in the next theorem. The arguments in the proof are standard and similar to the standard convergence analysis of the classical penalty methods, except that we make use of (i) the feasible point $\xfeas$ to guarantee that any limit point is feasible for \eqref{eq:P0}; (ii) the constraint qualification \eqref{CQ0} to guarantee the boundedness of ``Lagrange multipliers". For completeness, we include the proof in Appendix~\ref{appendix:B}.

\begin{theorem} \label{pm-convergence}
Suppose that Assumption~\ref{assum2} holds. Let $\{x^k\}$ be generated by the above penalty method for solving problem \eqref{eq:P0}. There hold:
\bi
\item[(i)] $\{x^k\}$ is bounded;
\item[(ii)] Any accumulation point $x^*$ of $\{x^k\}$ is a feasible point of problem \eqref{eq:P0}.
\item[(iii)] Suppose that $\{x^k\}_\cK \to x^*$ for some subsequence $\cK$ and that the constraint qualification \eqref{CQ0} holds at $x^*$.
Then $x^*$ is a KKT
point of problem \eqref{eq:P0}.
\ei
\end{theorem}


\section{Numerical simulations}\label{sec6}

In this section, we consider the problem of recovering a sparse solution of an underdetermined linear system from noisy measurements. In the literature, this is typically done via solving \eqref{eq:P00} or \eqref{eq:P-1} with a specific sparsity inducing function $\Phi$, e.g., the $\ell_1$ norm or the $\ell_{1/2}$ quasi-norm; see, for example, \cite{BergMPF08,BDE,CGWY,CXY} and references therein. Here, we propose using the model \eqref{eq:P00} (a special case of \eqref{eq:P0} with $S_1 = \R^n$ and $B$ being vacuous) with $\Phi(x) = \sum_{i=1}^n|x_i|^p$, $p = 1/2$. We solve this problem using our penalty method proposed in Subsection~\ref{sec:subalg}, which involves solving a sequence of subproblems in the form of \eqref{eq:P1}.
We benchmark our method against two other approaches:
\begin{enumerate}
    \item the solver SPGL1 \cite{BergMPF08} (Version 1.8) that solves \eqref{eq:P00} with $\Phi(x) = \|x\|_1$;
  \item the quadratic penalty method that solves \eqref{eq:P-1} with $\Phi(x) = \sum_{i=1}^n|x_i|^{1/2}$ and some suitable $\lambda > 0$.
\end{enumerate}
All codes are written in MATLAB,
and the experiments were performed in MATLAB version R2014a on a cluster with 32 processors (2.9 GHz each) and 252G RAM.

For our penalty method, we set $x^0 = e$, the vector of all ones, $\lambda_0=\mu_0=\epsilon_0=1$, $\rho = 2$ and $\theta = 1/\rho$. We also set $\xfeas = A^\dagger b$, which we take as an input to the algorithm and does not count this computation in our CPU time below. For the NPG method for solving the unconstrained subproblem \eqref{eq:P1-mu} at $\lambda = \lambda_k$ and $\mu = \mu_k$, we set $L_{\min} = 1$, $L_{\max} = 10^8$, $\tau = 2$, $c = 10^{-4}$, $M = 4$, $L^0_0=1$ and, for any $l\ge 1$,
\[
L_l^0 := \min\left\{\max\left\{\frac{[x^{k,l} - x^{k,l-1}]^T[\nabla f_{\lambda_k,\mu_k}(x^{k,l}) - \nabla f_{\lambda_k,\mu_k}(x^{k,l-1})]}{\|x^{k,l} - x^{k,l-1}\|^2},L_{\min}\right\},L_{\max}\right\}.
\]
The NPG method is terminated (at the $l$th inner iteration) when
\begin{equation*}
  \|{\rm Diag}(x^{k,l})\nabla f_{\lambda_k,\mu_k}(x^{k,l}) + p|x^{k,l}|^p\|_\infty\le \sqrt{\epsilon_k}\ \ {\rm and}\ \
  \frac{|F_{\lambda_k,\mu_k}(x^{k,l}) - F_{\lambda_k,\mu_k}(x^{k,l-1})|}{\max\{1,|F_{\lambda_k,\mu_k}(x^{k,l})|\}}\le \min\{\epsilon_k^2,10^{-4}\}.
\end{equation*}
Note that the first condition above means the first-order optimality condition \eqref{Lp1storder} is approximately satisfied. The penalty method itself is terminated when
\[
\max\left\{(\|Ax^k - b\|^2 - \sigma^2)_+, 0.01\epsilon_k\right\} \le 10^{-6},
\]
with the $\epsilon_{k+1}$ in step 2) of the penalty method updated as $\max\{\theta\epsilon_k,10^{-6}\}$ (instead of $\theta\epsilon_k$) in our implementation.

For the aforementioned SPGL1 \cite{BergMPF08}, we use the default settings.
For the quadratic penalty model \eqref{eq:P-1}, as discussed in our Example~\ref{example1}, there may be no $\lambda > 0$ so that the local minimizers of \eqref{eq:P-1} are closely related to
those of \eqref{eq:P00}. However, one can observe as $\lambda$ increases from $0$ to $\infty$, the residual $\|A\tilde x(\lambda) - b\|$ changes from $\|b\|$ to $0$, where
$\tilde x(\lambda)$ is an optimal solution of \eqref{eq:P-1}. Thus, a possibly best approximate solution to \eqref{eq:P0} offered by model \eqref{eq:P-1} appears to be the one corresponding to
the least $\lambda$ such that $\|A\tilde x(\lambda) - b\| \le \sigma$. However, such a $\lambda$ is typically unknown. Instead, we solve a sequence of problem \eqref{eq:P-1} along an increasing sequence of $\lambda$, and terminate when the approximate solution is approximately feasible for \eqref{eq:P00}. Specifically, we apply the same scheme described in our penalty method but with $H_\lambda$ in place of $F_{\lambda,\mu}$ and $\lambda \|Ax - b\|^2$ in place of $f_{\lambda,\mu}$, and we use exactly the same parameter settings as above. For ease of reference, we
call this approach and our proposed penalty method as ``Inexact Penalty'' and ``Exact Penalty'' methods, respectively.

We consider randomly generated instances. First, we generate a matrix $\tilde A\in \R^{K\times N}$ with i.i.d. standard Gaussian entries.
The matrix $A$ is then constructed so that its rows form an orthonormal basis for the row space of $\tilde A$.
Next, we generate a vector $v\in \R^T$ with i.i.d. standard Gaussian entries. We choose an index set $I$ of size $T$ at random
and define a vector $\hat x\in \R^N$ by setting $\hat x_I = v$ and $\hat x_{\bar I} = 0$.
The measurement $b$ is then set to be $A\hat x + \delta \xi$ for some $\delta > 0$, with each entry of $\xi$ following again the standard Gaussian distribution.
Finally, we set $\sigma = \delta\|\xi\|$ so that the resulting feasible set will contain the sparse vector $\hat x$.\footnote{In our simulations, all random instances satisfy $\|b\| > \sigma$, which implies that the origin is excluded from the feasible region of the problem.}

In our tests below, we set $(K,N,T) = (120i,512i,20i)$ for each $i = 12,14,...,30$ and generate $10$ random instances for each such $(K,N,T)$ as described above. The computational results reported are averaged over the $10$ instances, and they are reported in Tables~\ref{t1}, \ref{t2} and \ref{t3}, which present results for $\delta = 10^{-2}$, $5\times 10^{-3}$ and $10^{-3}$, respectively. For all three methods, we report the number of nonzero entries ({\bf nnz}) in the approximate solution $x$ obtained, computed using the MATLAB function $\sf nnz$, the recovery error ({\bf err}) $\|x - \hat x\|$, and the CPU time in seconds. We also report the function value $\Phi(x)$ at termination ({\bf fval}) for the penalty methods.
One can observe from these tables that our penalty method usually produces sparser solutions with smaller recovery errors than
the other two approaches though it is in general slower than SPGL1. Moreover, in contrast with the method ``Inexact Penalty'', our penalty
method achieves smaller objective values. 
These phenomena indeed reflect the intrinsic advantage of our approach.

\begin{table}[h!]
\caption{Comparing the penalty method and SPGL1, $\delta = 10^{-2}$}\label{t1} \normalsize
{\hspace*{-1.6cm}\footnotesize
\begin{tabular}{|c|c|c||c|c|c||c|c|c|c||c|c|c|c|c|}  \hline
\multicolumn{3}{|c||}{Data} & \multicolumn{3}{c||}{SPGL1} & \multicolumn{4}{c||}{Inexact Penalty} & \multicolumn{4}{c|}{Exact Penalty}
\\
$K$ & $N$ & $T$ & ${\bf nnz}$ & {\bf err} & CPU & ${\bf fval}$ & ${\bf nnz}$ & {\bf err} & CPU & ${\bf fval}$ & ${\bf nnz}$ & {\bf err} & CPU 
\\ \hline
          1440 &          6144 &           240 &       719 & 1.2e+00 &       0.69 &2.89e+02 &        859 & 9.2e-01 &      15.27 &1.90e+02 &        219 & 5.1e-01 &       5.08 \\
          1680 &          7168 &           280 &       837 & 1.3e+00 &       0.80 &3.38e+02 &        998 & 1.0e+00 &      17.44 &2.23e+02 &        257 & 5.5e-01 &       5.79 \\
          1920 &          8192 &           320 &       943 & 1.4e+00 &       1.06 &3.87e+02 &       1139 & 1.1e+00 &      23.85 &2.57e+02 &        294 & 5.7e-01 &       7.37\\
          2160 &          9216 &           360 &      1050 & 1.5e+00 &       1.27 &4.35e+02 &       1290 & 1.1e+00 &      28.91 &2.87e+02 &        330 & 6.1e-01 &      10.37 \\
          2400 &         10240 &           400 &      1188 & 1.6e+00 &       1.53 &4.82e+02 &       1430 & 1.2e+00 &      34.38 &3.17e+02 &        366 & 6.6e-01 &      11.80 \\
          2640 &         11264 &           440 &      1266 & 1.6e+00 &       1.87 &5.31e+02 &       1568 & 1.3e+00 &      43.91 &3.49e+02 &        402 & 6.7e-01 &      13.98 \\
          2880 &         12288 &           480 &      1404 & 1.7e+00 &       2.20 &5.78e+02 &       1712 & 1.3e+00 &      51.89 &3.81e+02 &        439 & 7.0e-01 &      20.21 \\
          3120 &         13312 &           520 &      1500 & 1.7e+00 &       2.79 &6.28e+02 &       1849 & 1.4e+00 &      64.28 &4.15e+02 &        474 & 7.4e-01 &      21.67 \\
          3360 &         14336 &           560 &      1656 & 1.8e+00 &       2.92 &6.75e+02 &       2000 & 1.4e+00 &      64.65 &4.46e+02 &        514 & 7.7e-01 &      24.77 \\
          3600 &         15360 &           600 &      1755 & 1.9e+00 &       3.28 &7.24e+02 &       2137 & 1.5e+00 &      75.72 &4.78e+02 &        546 & 7.9e-01 &      25.12 \\ \hline
\end{tabular}
}
\end{table}

\begin{table}[h!]
\small
\caption{Comparing the penalty method and SPGL1, $\delta = 5\times 10^{-3}$}\label{t2} \normalsize
{\hspace{-1.6cm}\footnotesize
\begin{tabular}{|c|c|c||c|c|c||c|c|c|c||c|c|c|c|c|}  \hline
\multicolumn{3}{|c||}{Data} & \multicolumn{3}{c||}{SPGL1} & \multicolumn{4}{c||}{Inexact Penalty} & \multicolumn{4}{c|}{Exact Penalty}
\\
$K$ & $N$ & $T$ & ${\bf nnz}$ & {\bf err} & CPU & ${\bf fval}$ & ${\bf nnz}$ & {\bf err} & CPU & ${\bf fval}$ & ${\bf nnz}$ & {\bf err} & CPU 
\\ \hline
          1440 &          6144 &           240 &       727 & 6.1e-01 &       0.78 &2.54e+02 &        738 & 4.4e-01 &      10.40 &1.94e+02 &        228 & 2.5e-01 &       4.68  \\
          1680 &          7168 &           280 &       827 & 6.7e-01 &       0.97 &2.94e+02 &        865 & 4.9e-01 &      13.20 &2.23e+02 &        266 & 2.7e-01 &       5.67  \\
          1920 &          8192 &           320 &       960 & 7.2e-01 &       1.31 &3.39e+02 &        988 & 5.3e-01 &      18.56 &2.57e+02 &        304 & 2.9e-01 &       7.93 \\
          2160 &          9216 &           360 &      1068 & 7.5e-01 &       1.58 &3.83e+02 &       1104 & 5.5e-01 &      23.95 &2.92e+02 &        342 & 3.0e-01 &      11.55 \\
          2400 &         10240 &           400 &      1195 & 7.9e-01 &       1.89 &4.28e+02 &       1230 & 5.8e-01 &      29.73 &3.26e+02 &        378 & 3.2e-01 &      11.47  \\
          2640 &         11264 &           440 &      1320 & 8.4e-01 &       2.35 &4.66e+02 &       1352 & 6.1e-01 &      35.31 &3.54e+02 &        416 & 3.5e-01 &      15.63  \\
          2880 &         12288 &           480 &      1422 & 8.7e-01 &       2.78 &5.10e+02 &       1472 & 6.4e-01 &      40.89 &3.88e+02 &        455 & 3.6e-01 &      16.76  \\
          3120 &         13312 &           520 &      1580 & 9.3e-01 &       3.23 &5.54e+02 &       1600 & 6.7e-01 &      46.70 &4.22e+02 &        496 & 3.7e-01 &      20.15  \\
          3360 &         14336 &           560 &      1668 & 9.5e-01 &       3.43 &5.94e+02 &       1715 & 6.9e-01 &      52.10 &4.53e+02 &        530 & 3.8e-01 &      24.81  \\
          3600 &         15360 &           600 &      1794 & 9.8e-01 &       3.89 &6.40e+02 &       1841 & 7.2e-01 &      54.26 &4.87e+02 &        570 & 3.9e-01 &      26.36  \\
\hline
\end{tabular}
}
\end{table}

\begin{table}[h!]
\small
\caption{Comparing the penalty method and SPGL1, $\delta = 10^{-3}$}\label{t3} \normalsize
{\hspace{-1.6cm}\footnotesize
\begin{tabular}{|c|c|c||c|c|c||c|c|c|c||c|c|c|c|c|}  \hline
\multicolumn{3}{|c||}{Data} & \multicolumn{3}{c||}{SPGL1} & \multicolumn{4}{c||}{Inexact Penalty} & \multicolumn{4}{c|}{Exact Penalty}
\\
$K$ & $N$ & $T$ & ${\bf nnz}$ & {\bf err} & CPU & ${\bf fval}$ & ${\bf nnz}$ & {\bf err} & CPU & ${\bf fval}$ & ${\bf nnz}$ & {\bf err} & CPU 
\\ \hline
          1440 &          6144 &           240 &       743 & 1.3e-01 &       1.24 &2.02e+02 &        345 & 6.1e-02 &       5.63 &1.95e+02 &        236 & 4.9e-02 &       6.49 \\
          1680 &          7168 &           280 &       880 & 1.4e-01 &       1.47 &2.38e+02 &        396 & 6.5e-02 &       6.35 &2.30e+02 &        275 & 5.5e-02 &       6.75 \\
          1920 &          8192 &           320 &       995 & 1.4e-01 &       1.93 &2.74e+02 &        460 & 7.0e-02 &       8.21 &2.64e+02 &        315 & 5.8e-02 &       8.84 \\
          2160 &          9216 &           360 &      1120 & 1.5e-01 &       2.08 &3.08e+02 &        511 & 7.3e-02 &       9.36 &2.97e+02 &        354 & 6.1e-02 &      11.23 \\
          2400 &         10240 &           400 &      1232 & 1.6e-01 &       2.59 &3.41e+02 &        573 & 7.9e-02 &      11.51 &3.28e+02 &        393 & 6.4e-02 &      13.60 \\
          2640 &         11264 &           440 &      1410 & 1.7e-01 &       2.96 &3.73e+02 &        631 & 8.3e-02 &      13.78 &3.59e+02 &        431 & 6.8e-02 &      17.26\\
          2880 &         12288 &           480 &      1476 & 1.7e-01 &       3.71 &4.08e+02 &        687 & 8.6e-02 &      15.82 &3.93e+02 &        472 & 7.0e-02 &      18.00 \\
          3120 &         13312 &           520 &      1613 & 1.9e-01 &       4.13 &4.42e+02 &        742 & 9.0e-02 &      18.29 &4.26e+02 &        511 & 7.5e-02 &      23.66 \\
          3360 &         14336 &           560 &      1720 & 1.9e-01 &       4.81 &4.78e+02 &        803 & 9.4e-02 &      21.97 &4.61e+02 &        551 & 7.7e-02 &      28.99 \\
          3600 &         15360 &           600 &      1857 & 2.0e-01 &       5.17 &5.07e+02 &        863 & 9.8e-02 &      24.26 &4.87e+02 &        591 & 7.9e-02 &      27.44 \\
\hline
\end{tabular}
}
\end{table}

\section{Concluding remarks}\label{sec7}

Optimization models in finding sparse solutions to underdetermined systems of linear equations have stimulated development in signal processing and image sciences.
The constrained optimization model \eqref{eq:P00} and regularization model \eqref{eq:P-1} have been widely used in this context when the data has noise.
The existence of a regularization parameter $\lambda$ such that problems \eqref{eq:P00} and \eqref{eq:P-1} have a common global minimizer is known if
the function $\Phi$ is convex.  However,  when $\Phi$ is nonconvex, such a $\lambda$ does not always exist, as shown in Example~\ref{example1}.
In this paper, we proposed a new penalty model \eqref{eq:P1} for the more general problem \eqref{eq:P0} where $\Phi$ can be nonconvex nonsmooth, perhaps even non-Lipschitz.
We studied the existence of exact penalty parameters for \eqref{eq:P0} regarding local minimizers, stationary points and $\epsilon$-minimizers.
Moreover, we proposed a new penalty method which solves the constrained problem \eqref{eq:P0} by solving a sequence of \eqref{eq:P1} via the proximal gradient algorithm,
with an update scheme for the penalty
parameters. We also proved the convergence of the penalty method to a KKT point of \eqref{eq:P0}.
Preliminary numerical results showed that our penalty method is efficient for finding sparse solutions to underdetermined systems.


\appendix

\section{Convergence of a nonmonotone proximal gradient method}\label{sec:NPG}

In this appendix, we consider an algorithm for solving the following optimization problem
\beq \label{f+phi}
\min\limits_x F(x) := f(x) + P(x),
\eeq
where $f$ and $P$ satisfy the following assumptions:
\begin{assumption} \label{assump-fphi}
\bi
\item[(i)] $f$ is continuously differentiable in $\cU(x^0;\Delta)$ for some
$x^0\in {\rm dom}\,P:= \{x: P(x) < \infty\}$ and $\Delta>0$, and moreover, there exists some $L_f>0$ such that
\beq \label{lipschitz}
\|\nabla f(x)-\nabla f(y)\| \le L_f \|x-y\|, \quad\quad \forall x,y \in \cU(x^0;\Delta),
\eeq
where
\beqas
 \cU(x^0,\Delta)  &:= & \left\{x: \|x-z\| \le \Delta \ \mbox{for some} \ z \in \cS(x^0) \right\}, \\ [6pt]
 \cS(x^0)  &:=&  \left\{x\in\R^n: \ F(x) \le F(x^0)\right\}.
\eeqas
\item[(ii)] $P$ is a proper lower semicontinuous function in $\R^n$.
\item[(iii)] $F$ is bounded below and uniformly continuous in $\cS(x^0)$.
\item[(iv)] The quantities $A$, $B$ and $C$ defined below are finite:
\beq \label{ABC}
A := \sup\limits_{x \in \cS(x^0)} \|\nabla f(x)\|,  \quad  B := \sup\limits_{x \in \cS(x^0)} P(x),
\quad C := \inf\limits_{x \in \R^n} P(x).
\eeq
\ei
\end{assumption}

The algorithm we consider is a nonmonotone proximal gradient method, presented as follows.

\gap

\noindent
{\bf Algorithm 1: Nonmonotone proximal gradient (NPG) method for \eqref{f+phi}} \\ [5pt]
Let $x^0$ be given in Assumption \ref{assump-fphi}. Choose $L_{\max} \ge L_{\min} >0$,
$\tau>1$, $c>0$ and an integer $M \ge 0$ arbitrarily. Set $k=0$.
\begin{itemize}
\item[1)] Choose $L^0_k \in [L_{\min}, L_{\max}]$ arbitrarily. Set $L_k = L^0_k$.
\bi
\item[1a)] Solve the subproblem
\beq \label{prox-subprob}
u \in \Argmin_x \left\{\langle \nabla f(x^k), x-x^k \rangle +
\frac{L_k}{2}\|x-x^k\|^2 + P(x)\right\}.\footnote{This problem has at least one
optimal solution due to Assumption \ref{assump-fphi} (ii) and (iv).}
\eeq
\item[1b)] If
\beq \label{descent}
F(u) \le \max_{[k-M]_+ \le i \le k} F(x^i) - \frac{c}{2} \|u-x^k\|^2
\eeq
is satisfied, then go to step 2).
\item[1c)] Set $L_k \leftarrow \tau L_k$ and go to step 1a).
\ei
\item[2)]
Set $x^{k+1}\leftarrow u$, $\bar L_k\leftarrow L_k$, $k \leftarrow k+1$ and go to step 1).
\end{itemize}
\noindent
{\bf end}

Although the NPG method has been analyzed in \cite{WrNoFi09}, the analysis there relies on the
assumption that $\nabla f$ is globally Lipschitz continuous in $\R^n$. In our Assumption~\ref{assump-fphi},
$\nabla f$ is, however, not necessarily globally Lipschitz continuous and thus the analysis in \cite{WrNoFi09}  does not apply directly to problem \eqref{f+phi}.
We next show that the NPG method is still convergent for problem \eqref{f+phi} under Assumption~\ref{assump-fphi}.

\begin{proposition} \label{inner-convergence}
Let $x^k$ be the approximate solution generated at the end of the
$k$th iteration, and let
\beq \label{underL}
\bar L := \max\{L_{\max},\tau\underline L, \tau(L_f +c)\}, \quad\quad \underline L :=  \frac{2A\Delta+2(B-C)}{\Delta^2},
\eeq
where $A$, $B$, $C$ and $\Delta$ are given in Assumption \ref{assump-fphi}.
Under Assumption \ref{assump-fphi}, there hold:
\begin{itemize}
\item[(i)] $x^{k+1}$ is well defined and $F(x^{k+1}) \le F(x^0)$ for all $k \ge 0$;
\item[(ii)] $\bar L_k$ is well defined and satisfies $\bar L_k \le \bar L$ for all $k \ge 0$.
\item[(iii)] For each $k \ge 0$, the inner termination criterion \eqref{descent} is satisfied after
at most
\[
\left\lfloor \frac{\log( \bar L)-\log(L_{\min})}{\log \tau} +1\right\rfloor
\]
inner iterations.
\end{itemize}
\end{proposition}

\begin{proof}
For convenience, whenever $x^k$ is well defined with $F(x^k)\le F(x^0)$, set
\beq \label{xkL}
x^{k+1}(L) \in  \Argmin_{x \in \R^n} \left\{\langle \nabla f(x^k), x-x^k \rangle +
\frac{L}{2}\|x-x^k\|^2 + P(x)\right\} \quad\quad \forall L>0.
\eeq
By \eqref{xkL}, one can then observe that
\[
\langle \nabla f(x^k), x^{k+1}(L)-x^k \rangle + P(x^{k+1}(L)) + \frac{L}{2} \|x^{k+1}(L)-x^k\|^2 \le P(x^k),
\]
which along with \eqref{ABC}  yields
\[
\frac{L}{2} \|x^{k+1}(L)-x^k\|^2 - \|\nabla f(x^k)\| \|x^{k+1}(L)-x^k\| +C - P(x^k) \le 0.
\]
Hence, we obtain that
\beq \label{distance}
\|x^{k+1}(L)-x^k\| \le  \frac{\|\nabla f(x^k)\|+\sqrt{\|\nabla f(x^k)\|^2+2L(P(x^k)-C)}}{L}.
\eeq

We now prove statements (i) and (ii) by induction. Indeed, for $k=0$, we know that $x^0\in \cS(x^0)$.
Using this relation, \eqref{ABC} and \eqref{distance} with $k=0$, one can have
\[
\|x^1(L)-x^0\| \le  \frac{A+\sqrt{A^2+2L(B-C)}}{L}.
\]
In view of this inequality and \eqref{underL}, it is not hard to verify that
\[
\|x^1(L)-x^0\| \le \Delta, \quad\quad \forall L \ge \underline L.
\]
Using this relation and \eqref{lipschitz}, we have
\[
f(x^1(L)) \le f(x^0) + \langle \nabla f(x^0), x^1(L)-x^0\rangle + \frac{L_f}{2} \|x^1(L)-x^0\|^2, \quad\quad \forall
L \ge \underline L.
\]
It follows from this relation and \eqref{xkL} that for all $L \ge \underline L$,
\[
\ba{rl}
&F(x^1(L)) = f(x^1(L)) + P(x^1(L)) \\ [8pt]
&\le f(x^0) + \langle \nabla f(x^0), x^1(L)-x^0\rangle + \frac{L_f}{2} \|x^1(L)-x^0\|^2 + P(x^1(L)) \\ [8pt]
& =  f(x^0) + \langle \nabla f(x^0), x^1(L)-x^0\rangle + \frac{L}{2} \|x^1(L)-x^0\|^2 + P(x^1(L)) + \frac{L_f-L}{2} \|x^1(L)-x^0\|^2 \\ [8pt]
&\le  f(x^0) + P(x^0) + \frac{L_f-L}{2} \|x^1(L)-x^0\|^2 \ = \ F(x^0) + \frac{L_f-L}{2} \|x^1(L)-x^0\|^2,
\ea
\]
where the second inequality follows from \eqref{xkL}. Using this relation, one can immediately
observe that
\begin{equation}\label{fineq}
F(x^1(L)) \le F(x^0) - \frac{c}{2}\|x^1(L)-x^0\|^2, \quad\quad \forall L \ge \hat L,
\end{equation}
where
\[
\hat L :=  \max\{\underline L, L_f +c\}.
\]
This shows that \eqref{descent} must be satisfied after finitely many inner iterations. Moreover, from the definition of $\bar L_0$, we
must have  either $\bar L_0 =L^0_0$ or $\bar L_0/\tau < \hat L$. This together with $L^0_0 \le L_{\max}$ implies
$\bar L_0 \le \max\{L_{\max},\tau \hat L\}$, and hence statement (ii) holds for $k=0$. We also see from \eqref{fineq} that $F(x^1) = F(x^1(\bar L_0))\le F(x^0)$. Hence, statement (i) also holds for $k = 0$.

We now suppose that statements (i) and
(ii) hold for all $k \le K$ for some $K \ge 0$. It remains to show that they also hold for $k=K+1$. Indeed, using the
induction hypothesis, we have $x^{K+1} \in \cS(x^0)$. In view of
this relation and a similar argument as for $k=0$, one can show that statement (ii) holds for $k=K+1$.
By the induction hypothesis, we know that $F(x^{k+1}) \le F(x^0)$ for all $k \le K$. Using this
relation and \eqref{descent} with $k=K+1$, one can conclude that $F(x^{K+2}) \le F(x^0)$ and hence statement (i) holds for $k=K+1$.
This completes the induction.

Finally we prove statement (iii). Let $n_k$ denote the total number of inner iterations executed at the $k$th outer iteration.
One can observe that
\[
L_{\min} \tau^{n_k-1} \le L^0_k \tau^{n_k-1} = \bar L_k.
\]
The conclusion then immediately follows from this relation and statement (ii).
\end{proof}

\gap

We end our discussion with a convergence result for the NPG method, which can be proved similarly as in \cite[Lemma~4]{WrNoFi09}.

\begin{theorem} \label{outer-convergence}
Let $x^k$ be the approximate solution generated at the end of
the $k$th iteration. Under Assumption \ref{assump-fphi}, there holds $\|x^{k+1}-x^k\| \to 0$
as $k \to \infty$.
\end{theorem}

\section{Proof of Theorem~\ref{pm-convergence}}\label{appendix:B}

In this section, we present the proof of Theorem~\ref{pm-convergence}.

\begin{proof}
(i) By Proposition \ref{inner-convergence}, we know that $F_{\lambda_k,\mu_k}(x^k) \le F_{\lambda_k,\mu_k}(x^{k,0})$.
In addition, from step 1) of the above penalty method, one has $F_{\lambda_k,\mu_k}(x^{k,0}) \le F_{\lambda_k,\mu_k}(\xfeas)$.
It then follows that $F_{\lambda_k,\mu_k}(x^k) \le F_{\lambda_k,\mu_k}(\xfeas)$. Using this relation along with \eqref{f-prop}
and the facts that $\|A\xfeas-b\| \le \sigma$ and $B\xfeas \le h$, one can have
\[
\Phi(x^k) \ \le \ F_{\lambda_k,\mu_k}(x^k) \ \le \ F_{\lambda_k,\mu_k}(\xfeas) = \Phi(\xfeas).
\]
Moreover, we also have $x^k\in S_1$ from the definition.
Hence, $\{x^k\}$ is bounded since $\Phi + \delta_{S_1}$ has bounded level sets.

(ii) Let $x^*$ be an accumulation point of $\{x^k\}$. Then there exists a subsequence
$\{x^k\}_\cK \to x^*$. Using $F_{\lambda_k,\mu_k}(x^k) \le F_{\lambda_k,\mu_k}(\xfeas)$, \eqref{f-prop} and
the definition of $F_{\lambda,\mu}$, we have
\[
\ba{lcl}
\lambda_k (\|Ax^k - b\|^2 - \sigma^2)_+ &+& \lambda_k \|(B x^k - h)_+\|_1\le f_{\lambda_k,\mu_k}(x^k) + \frac{\ell + 1}2\lambda_k \mu_k\\ [8pt]
& \le& \
F_{\lambda_k,\mu_k}(x^k) + \frac{\ell + 1}2\lambda_k \mu_k
 \le  F_{\lambda_k,\mu_k}(\xfeas) + \frac{\ell + 1}2\lambda_k \mu_k  \\ [8pt] &=& \  \Phi(\xfeas) + \frac{\ell + 1}2\lambda_k \mu_k.
\ea
\]
It then follows that
\[
(\|Ax^k - b\|^2 - \sigma^2)_+ +  \|(B x^k - h)_+\|_1 \le \ \frac{\Phi(\xfeas)}{\lambda_k} + \frac{\ell + 1}2\mu_k.
\]
Taking limits on both sides of this inequality as $k\in \cK \to \infty$, one has $(\|Ax^* - b\|^2 - \sigma^2)_+ \le 0$ and $\|(B x^* - h)_+\|_1\le 0$.
Hence $x^*$ is a feasible point of problem \eqref{eq:P0}.

(iii) Let $I_* := \{i:\; (Bx^* - h)_i = 0\}$. Then $(Bx^*)_i < h_i$ for all $i\notin I_*$ and we have
\[
\cN_{B\cdot \le h}(x^*) = \left\{\sum_{i\in I_*}y_i b_i:\; y\ge 0\right\},
\]
where $b_i$ denotes the column vector formed from the $i$th row of $B$.
Moreover, for all sufficiently large $k\in \cK$, we have $(Bx^k)_i < h_i$ for all $i\notin I_*$. Using this and \eqref{h-deriv}, we have $w^k_i := h'_{\lambda_k,\mu_k}([Bx^k - h]_i) = 0$ for $i\notin I_*$ and all sufficiently large $k$. This together with \eqref{inner-cond} and \eqref{f-deriv} implies that for all $k\in\cK$ sufficiently large, there exists $\xi^k\in \partial (\Phi + \delta_{S_1})(x^k)$ so that
\begin{equation}\label{relation1}
\left\|2 h'_{\lambda_k,\mu_k}(\|Ax^k - b\|^2 - \sigma^2) A^T(Ax^k-b) + \xi^k + \sum_{i\in I_*}w_i^kb_i\right\| \le \ \epsilon_k.
\end{equation}
We consider two different cases: $\|Ax^* - b\| < \sigma$ or $\|Ax^* - b\| = \sigma$.

{\bf Case 1.} Suppose first that $x^*$ satisfies $\|Ax^* - b\| < \sigma$. Then $\|Ax^k-b\|<\sigma$ for all sufficiently large $k\in\cK$.
Using this relation and \eqref{h-deriv}, we have $h'_{\lambda_k,\mu_k}(\|Ax^k - b\|^2 - \sigma^2)=0$ for all sufficiently
large $k\in\cK$. Hence, the relation \eqref{relation1} reduces to
\begin{equation}\label{relation2}
\left\|\xi^k + \sum_{i\in I_*}w_i^kb_i\right\| \le \ \epsilon_k.
\end{equation}
We suppose to the contrary that $\|\xi^k\|$ is unbounded. Without loss of generality, assume that $\{\|\xi^k\|\}_{\cK}\to \infty$ and that $\lim\limits_{k\in \cK}\frac{\xi^k}{\|\xi^k\|} = \xi^*$ for some $\xi^*$. Divide both sides of \eqref{relation2} by $\|\xi^k\|$ and pass to the limit, making use of $\epsilon_k\to 0$, \eqref{outersemi} and the closeness of the conical hull of the finite set $\{b_i:\;i\in I_*\}$, we see further that $\xi^*\in \partial^{^\infty}\!\!(\Phi + \delta_{S_1})(x^*)$ and
\[
-\xi^* \in \left\{\sum_{i\in I_*}y_i b_i:\; y\ge 0\right\} = \cN_{B\cdot \le h}(x^*) = \cN_{S_2}(x^*),
\]
where the second equality follows from the fact that $\|Ax^* - b\|< \sigma$. Since $\|\xi^*\| = 1$, this is a contradiction to \eqref{CQ0}. This shows that $\|\xi^k\|$ is bounded. By passing to the limit along a convergent subsequence in \eqref{relation2}, using \eqref{outersemi} and the closedness of finitely generated cones, we obtain
\[
0 \in \partial (\Phi + \delta_{S_1})(x^*) + \left\{\sum_{i\in I_*}y_i b_i:\; y\ge 0\right\} = \partial (\Phi + \delta_{S_1})(x^*) + \cN_{S_2}(x^*),
\]
i.e., $x^*$ is a KKT point of \eqref{eq:P0}.

{\bf Case 2.} Suppose now that $x^*$ satisfies $\|Ax^* - b\| = \sigma$. Observe from \eqref{h-deriv} that $h'_{\lambda_k,\mu_k}(\|Ax^k - b\|^2 - \sigma^2)\ge 0$ for all $k$. Let $t_k := 2h'_{\lambda_k,\mu_k}(\|Ax^k - b\|^2 - \sigma^2)$ for notational simplicity, and suppose for
contradiction that the sequence $\{\|\xi^k\|\}_\cK$ is unbounded. Without loss of generality,
assume that $\{\|\xi^k\|\}_\cK \to \infty$. It follows from \eqref{relation1} that
\beq \label{inner-cond1}
\left\| \frac{t_k}{\|\xi^k\|}A^T(Ax^k-b) + \frac{1}{\|\xi^k\|}\xi^k + \sum_{i\in I_*}\frac{w_i^k}{\|\xi^k\|}b_i\right\| \ \le \ \frac{\epsilon_k}{\|\xi^k\|}.
\eeq
We claim that $\{\frac{t_k}{\|\xi^k\|}\}_{\cK}$ is bounded. Suppose to the contrary and without loss of generality that $\{\frac{t_k}{\|\xi^k\|}\}_{\cK}\to \infty$. Dividing both sides of \eqref{inner-cond1} by $\frac{t_k}{\|\xi^k\|}$, passing to the limit and using the closedness of finitely generated cones, we see that
\begin{equation}\label{forcontra}
0\in A^T(Ax^* - b) + \cN_{B\cdot\le h}(x^*).
\end{equation}
This means that $x^*$ is an optimal solution of the problem
\begin{equation*}
  \begin{array}{rl}
    \min\limits_{x} & \frac12 \|Ax - b\|^2\\
    {\rm s.t.} & Bx \le h.
  \end{array}
\end{equation*}
Since $\|Ax^* - b\| = \sigma$, this contradicts our assumption that there is $x_0\in S$ with $\|Ax_0 - b\| < \sigma$. This contradiction shows that $\{\frac{t_k}{\|\xi^k\|}\}_{\cK}$ is bounded. By passing to a further subsequence if necessary, we may now assume without loss of generality that
\[
\lim_{k\in \cK}\frac{t_k}{\|\xi^k\|} = t_*,\ \ {\rm and}\ \ \lim_{k\in \cK}\frac{\xi^k}{\|\xi^k\|} = \xi^*.
\]
Note that $\xi^*\in \partial^{^\infty}\!\!(\Phi + \delta_{S_1})(x^*)$ due to \eqref{outersemi}.
Taking limit on both sides of \eqref{inner-cond1} along this subsequence and making use again of the closedness of finitely generated cones, we see further that
\begin{equation}\label{haha}
-\xi^* \in t_* A^T(Ax^* - b) + \left\{\sum_{i\in I_*}y_i b_i:\; y\ge 0\right\}\subseteq \cN_{\|A\cdot - b\|\le \sigma}(x^*) + \cN_{B\cdot \le h}(x^*) = \cN_{S_2}(x^*),
\end{equation}
where the set inclusion follows from the fact that $\|Ax^* - b\| = \sigma$ and the existence of $x_0\in S$ with $\|Ax_0 - b\| < \sigma$; this latter condition also gives the last equality in \eqref{haha}. Since $\|\xi^*\| = 1$, the relation \eqref{haha} together with $\xi^*\in \partial^{^\infty}\!\!(\Phi + \delta_{S_1})(x^*)$ contradicts \eqref{CQ0}. Thus, the sequence $\{\|\xi^k\|\}_\cK$ is bounded.

Next, we claim that $\{t_k\}_{\cK}$ is bounded. Assume again to the contrary that $\{t_k\}_{\cK}$ is unbounded and assume without loss of generality that $\{t_k\}_{\cK}\to \infty$.
From \eqref{relation1}, we have
\beq \label{inner-cond2}
\left\|A^T(Ax^k-b) + \frac{1}{t_k}\xi^k + \sum_{i\in I_*}\frac{w_i^k}{t_k}b_i\right\| \ \le \ \frac{\epsilon_k}{t_k}.
\eeq
Passing to the limit in \eqref{inner-cond2} and using the boundedness of $\xi^k$ as well as the closedness of finitely generated cones, we arrive at \eqref{forcontra}. A contradiction can then be derived similarly as before. Thus, we conclude that $\{t_k\}_{\cK}$ is bounded.

Let $\pi^*$ be an accumulation point of $\{t_k\}_\cK$. Without loss of generality,
assume that $\{t_k\}_\cK \to \pi^*$. Since $t_k\ge 0$
for all $k$, one has $\pi^* \ge 0$. Taking
limits on both sides of \eqref{relation1} as $k\in\cK \to \infty$, invoking \eqref{outersemi}, the boundedness of $\{\xi^k\}_{k\in \cK}$ and the closedness of finitely generated cones, one can see that
\[
0\in \pi^* A^T(Ax^* - b) + \partial (\Phi + \delta_{S_1})(x^*) + \cN_{B\cdot \le h}(x^*) \subseteq \partial (\Phi + \delta_{S_1})(x^*) + \cN_{S_2}(x^*).
\]
This shows that $x^*$ is a KKT point of \eqref{eq:P0}.
\end{proof}

\end{document}